\documentclass[journal]{IEEEtran}
\usepackage{amsmath,amsthm,amssymb,mathrsfs}
\usepackage{xcolor}
\usepackage{multirow}
\usepackage{pifont}
\usepackage{graphicx}
\usepackage{cite}

\newcommand{\ACCalgo}{\texttt{C2Opt}}
\newcommand{\DMDc}{\texttt{DMDc}}

\usepackage{tikz}
\usetikzlibrary{arrows.meta, plotmarks, decorations.text}
 
\newtheorem{corr}{Corollary}
\newtheorem{thm}{Theorem}
\newtheorem{defn}{Definition}
\newtheorem{lem}{Lemma}
\newtheorem{rem}{Remark}
\newtheorem{assum}{Assumption}
\newtheorem{prop}{Proposition}
\newtheorem{prob}{Problem}

\usepackage{algorithmic,algorithm}

\newcommand{\Nint}[1]{[#1]}
\newcommand{\diam}{\mathrm{diam}}

\newcommand{\datapoints}{(x_t,u_t,x_{t+1},C(z_t, u_t),
    \nabla C(z_t, u_t))}
\newcommand{\datapointsi}{(x_i,u_i,x_{i+1},C(z_i, u_i),
    \nabla C(z_i, u_i))}
\newcommand{\datapointsN}{{\{\datapoints\}}_{t=0}^{N-1}}
\newcommand{\dataSet}[1]{\mathscr{D}_{#1}}
\newcommand{\ProbInst}[2]{$\operatorname{ProbInst}(#1,F,C,
#2)$}
\newcommand{\ProbInstGen}{\ProbInst{\Phi}{\dataSet{N}}}
\newcommand{\IdentityMap}{\mathbb{I}}
\newcommand{\ProbInstZero}{\ProbInst{\IdentityMap}{\emptyset}}
\newcommand{\mt}{m_{t,\alpha,\beta}}
\newcommand{\mtprime}{m_{t,\alpha',\beta'}}
\newcommand{\GCmax}{{\|\nabla C\|}_{\max}}
\newcommand{\GapC}{C^\mathrm{gap}_{t-1}}

\newcommand{\dist}{\mathrm{dist}}

\makeatletter
\define@key{Gin}{Trim}
            {\let\Gin@viewport@code\Gin@trim\expandafter\Gread@parse@vp#1 \\}
\makeatother

\title{On-the-fly control of unknown nonlinear systems with sublinear regret}
\author{Abraham P. Vinod, Arie Israel, and Ufuk 
    Topcu
    \thanks{This material is based on work partly
        supported by the DARPA Assured Autonomy Program, Air Force Office of Scientific Research
        (FA9550-19-1-0005) and National Aeronautics and Space
        Administration (80NSSC19K0209).\newline
        \indent A. Vinod is with Mitsubishi Electric Research Laboratories (MERL), 02139 Cambridge, MA, USA; e-mail:
abraham.p.vinod@ieee.org.\newline
        \indent A.  Israel is with the Department of
        Mathematics, and U. Topcu is with the Department of
        Aerospace Engineering and Engineering Mechanics at the University of Texas at Austin,
    Austin, TX, 78712 USA; email: arie@math.utexas.edu,
utopcu@utexas.edu}
}

\date{}

\begin{document}
\maketitle

\begin{abstract}
We study the problem of data-driven, constrained control of
unknown nonlinear dynamics from a single ongoing and finite-horizon trajectory. We consider a one-step optimal control problem with a smooth, black-box objective, typically a composition of a known cost function and the unknown dynamics. We investigate an on-the-fly control paradigm, i.e., at each time step, the evolution of the dynamics and the first-order information of the cost are provided only for the executed control action. We propose an optimization-based control algorithm that iteratively minimizes a data-driven surrogate function for the unknown objective. We prove that the proposed approach incurs sublinear cumulative regret (step-wise suboptimality with respect to an optimal one-step controller) and is worst-case optimal among a broad class of data-driven control algorithms. We also present tractable reformulations of the approach that can leverage off-the-shelf solvers for efficient implementations.
\end{abstract}

\maketitle

\section{Introduction}

Optimal control of unknown dynamical systems using limited data and feedback is necessary for recovery after a catastrophic failure in physical systems. 
We focus on an extreme case of data-driven control, where
data from only a single finite-horizon trajectory is
available for control decision. Additionally, we consider
the setting of \emph{on-the-fly control} or \emph{bandit
feedback}~\cite{ornik2019myopic,djeumou2020fly,krause2011contextual},
where feedback regarding the optimal control problem is provided only at the executed state-action pair at each time step. 
The effect of all other admissible
actions at the current state remains unknown. An example of
these challenging restrictions occurs when an aircraft suffers from catastrophic damage mid-flight causing unknown alterations to its dynamics, and we only have limited data and feedback to land the aircraft using on-the-fly control.

We focus on the one-step optimal control problem due to the
lack of a predictive model. 
We consider a sequential decision-making problem where the
controller must minimize a one-step cost function under
constraints at each time step. The cost function
assigns preferences over the next state and is \emph{a
priori unknown} since the dynamics of the system are
unknown. 

We focus on the case in which the cost function has a
\emph{Lipschitz continuous gradient}. 
In this setting, we can reason about the cost
function at a previously unseen state-action pair using the 
available data and the 
smoothness assumption.  This paper \emph{utilizes 
optimization to design a data-driven control algorithm for the one-step optimal control problem, and
provide rigorous bounds on the suboptimality of the
proposed approach}.

One approach to solving sequential decision-making problems is
via iterative minimization of a data-driven model of the
unknown one-step cost function (objective). In the machine
learning community, such approaches have been studied
under the umbrella of \emph{contextual
optimization}~\cite{krause2011contextual,chowdhury2017ICML,
valko2013finite,berkenkamp2016safe,bogunovic2020corruption}.
These approaches utilize Gaussian processes or reproducing
kernel Hilbert spaces to model the unknown objective using
data.
The decision at each iteration is then made via the
optimization of a surrogate function of the control action
(to be optimized at each time step) and the current state
(known at each time step). 
While the proposed approach follows a similar approach of
minimization of a specific data-driven surrogate function,
it overcomes three main challenges suffered by existing approaches to contextual optimization.

First, surrogate functions proposed by existing approaches
for contextual optimization are non-convex. Therefore,
existing approaches typically rely on generic nonlinear optimization solvers or
discretization-based approaches to (approximately)
minimize the surrogate
functions~\cite{krause2011contextual,berkenkamp2016safe,bogunovic2020corruption}. The local optimization-based approaches
invalidate the theoretical bounds on the suboptimality of
the control decisions
available for such approaches which assume that the
global optima of the surrogate functions are found at each
iteration, while the discretization-based approaches
apply only to low dimensions (typically up to
three-dimensional problems). In contrast, the proposed
approach can be implemented exactly via concave quadratic
programming, is amenable to convexification that is
lossless in certain cases, and does not
have a gap between theory and practice in terms of the
suboptimality of the control decisions. 

Second, existing approaches
for contextual optimization do
not easily accommodate side information of the unknown
objective function like convexity and monotonicity of the
unknown cost function when available.  
Incorporation of such high-level information can
significantly reduce the amount of data required by the control
algorithm by improving the accuracy of interpolation of the one-step
cost function~\cite{vinod2020convexified,djeumou2020fly,ornik2019myopic,ahmadi2020learning,ahmadi2020safe}.
The structure of the surrogate function used by the proposed
approach admits such side information. We
additionally demonstrate that side information can recover
parts of the unknown cost function exactly in certain cases,
even when the available data is severely limited.

Third, existing approaches rely
on computationally intensive hyperparameter tuning for
building accurate data-driven models of the unknown
objective~\cite{krause2011contextual,chowdhury2017ICML,
valko2013finite}. In contrast, the
proposed approach has only a single hyperparameter with
intuitive connections to the underlying one-step control
problem.

We also prove that the proposed approach is worst-case
optimal among the class of all sequential optimization-based
approaches that can tackle the one-step optimal control
problem.
We utilize the notion of \emph{average regret} to study
the optimality of the proposed
algorithm, inspired by~\cite{krause2011contextual}. Roughly speaking,
the \emph{average regret} is the 
suboptimality of the proposed algorithm with respect to the
\emph{a priori unknown} optimal one-step controller
averaged over the entire duration of the application of
control.  We
characterize an upper bound on the average 
regret for the proposed approach, and prove that
it matches an algorithm-independent lower bound on the
average cumulative regret. Additionally, these bounds show
that the proposed approach enjoys a sublinear cumulative
regret, similar to the contextual Gaussian process-lower
confidence bound (\texttt{CGP-LCB}) algorithm 
~\cite{krause2011contextual}.  Sublinear
cumulative regret demonstrates that the designed
controller competes effectively with the unknown optimal
mapping from states to actions. 

For the data-driven control problem of interest, several
approaches have also been proposed in the literature that
does not directly invoke contextual
optimization. Instead, they combine system identification with model predictive
control.  In~\cite{korda2018linear}, the authors use
the Koopman
theory to lift the unknown nonlinear dynamics to a higher
dimensional space where linear system identification is
performed. \texttt{SINDYc}~\cite{kaiser2018sparse} utilizes
a sparse regression over a library of nonlinear functions
for nonlinear system identification.
\texttt{DMDc}~\cite{proctor2016dynamic,korda2018linear} uses the spectral
properties of the collected data to obtain approximate
linear models. 
However, such approaches are expected to perform poorly with nonlinear systems under severe data limitations. 
Alternatively, adaptive
control literature of nonlinear systems uses model inference
to arrive at the dynamics with relatively low data~\cite{
chowdhary2014bayesian, calliess2014conservative}. However,
these approaches are typically computationally expensive and do not accommodate side
information easily. Recent research has provided empirical
evidence to the utility of side information when data is
severely
limited~\cite{vinod2020convexified,djeumou2020fly,ornik2019myopic,ahmadi2020learning,ahmadi2020safe}.
In their current form, none of the above
approaches have the combined advantages of the proposed
approach --- the ability to accommodate severe data scarcity
and additional (side) information like convexity and
monotonicity of the optimal control problem, the knowledge
of the gradients, the non-asymptotic, worst-case
optimality guarantees on regret, and tractable
implementations using off-the-shelf solvers.  

The main contributions of this paper are 1) a data-driven
and optimization-based approach for constrained on-the-fly
control of a system with unknown nonlinear dynamics from a
single finite-horizon trajectory, 2) an upper bound
on the incurred average regret for the proposed algorithm
as a function of the number of control time steps that
implies sublinear cumulative regret, and 3) a
matching, algorithm-independent lower bound on the average
regret (matches up to constants).  We
show that the proposed approach can be implemented using
convex or concave quadratic minimization programs, by
exploiting its connections with difference-of-convex
programming~\cite{an2005dc}. We demonstrate the efficacy of
the proposed approach in two nonlinear data-driven control
problems: drive a unicycle to a target location under
control constraints using only observed data, and land a
damaged aircraft.

\section{Preliminaries and problem statements}

Let $\Nint{t}$ denote the finite set of natural numbers up to $t\in
\mathbb{N}$, $ \mathcal{G}^N$ denote the Cartesian product
of a set $ \mathcal{G}$ with itself $N\in \mathbb{N}$ times,
$| \mathcal{G}|$ denote the cardinality of a set $
\mathcal{G}$, and $\|\cdot\|$
denotes the standard Euclidean norm.  Given $J :
\mathbb{R}^d \to \mathbb{R}$, $\ell(x;J,q) \triangleq J(q) + \nabla J(q)
\cdot (x - q)$ denotes its linear
approximant about a data point $(q, J(q))\in \mathbb{R}^d
\times \mathbb{R}$. 

\subsection{Smooth functions}

Given a set $ \mathcal{S}\subset \mathbb{R}^n$, a continuously differentiable function $f: \mathcal{S} \to \mathbb{R}$ has a \emph{Lipschitz continuous gradient}, if its gradient $\nabla f$ satisfies the property $\|\nabla f(y) - \nabla f(x)\|\leq L_f\|y - x\|$ for every $x,y \in \mathcal{S}$ and some $L_f > 0$~\cite{nesterov2018lectures}. Here, we refer to $L_f$ as the \emph{Lipschitz gradient constant}. We denote the smallest Lipschitz gradient constant as $K_f$, although $K_f$ is rarely known. For brevity, we will refer to functions with Lipschitz continuous gradient as \emph{smooth functions}.

For any $s\in \mathcal{S}$ and data $\left(q_i,f(q_i),\nabla f(q_i)\right)_{i\in\Nint{t}}$ obtained at points $q_i \in \mathcal{S}$, we have the following data-driven minorant $f_t^-: \mathcal{S} \to \mathbb{R}$ and majorant $f_t^+: \mathcal{S} \to \mathbb{R}$,
\begin{align}
    f_t^+(s)&\triangleq\min\limits_{i\in
        \Nint{t}}\left({\ell(s;f,q_i)+\frac{L_f}{2} \| s -
    q_i\|^2}\right)\label{eq:C_plus_defn},\\ 
    f_t^-(s)&\triangleq\max\limits_{i\in \Nint{t}}\left({\ell(s;f,q_i)-\frac{L_f}{2} \| s -
    q_i\|^2}\right)\label{eq:C_minus_defn}.
\end{align}
\begin{lem}[\textsc{Data-driven approximants for $f$}]\label{lem:approx}
For any $s\in \mathcal{S}$, $f_t^-(s)\leq f(s)\leq f_t^+(s)$, and the approximation error $f(s) - f_t^-(s)$ and $f_t^+(s) - f(s)$ are bounded from below by zero and bounded from above by $L_f \min_{i\in\Nint{t}}\|s - q_i\|^2$.
\end{lem}

We provide the proof of Lemma~\ref{lem:approx} in Appendix~\ref{app:lem_approx}.

\subsection{Data-driven, on-the-fly control of nonlinear systems}

Consider a discrete-time nonlinear dynamical system
\begin{align}
    x_{t+1} = F(x_t, u_t) \label{eq:nonlin_sys},
\end{align}
with state $x_t\in \mathcal{X}\subseteq
\mathbb{R}^{n_{\mathcal{X}}}$ and control input $u_t \in
\mathcal{U}\subset \mathbb{R}^{n_{\mathcal{U}}}$.  The
nonlinear dynamics $F: \mathcal{X}\times \mathcal{U} \to
\mathcal{X}$ are unknown, but the sets $\mathcal{X}$ and
$\mathcal{U}$ are known, convex, and compact. 
\begin{assum}[\textsc{Full state feedback}]\label{assum:nonoise}
    At any time instant $t\in \mathbb{N}$, given the current
    state $x_t$ and the applied control $u_t$, we know the
    next state $x_{t+1}$. 
\end{assum}

Next, we introduce contexts to simplify the process of
learning nonlinear functions from data. 
\begin{defn}
    We define context as a known nonlinear transformation of
    the state $z=\Phi(x)$, with the context map
    $\Phi: \mathbb{R}^{n_{\mathcal{X}}}\to
    \mathbb{R}^{n_{\mathcal{Z}}}$ for some
    $n_{\mathcal{Z}}\in \mathbb{N}$, and
    $n_{\mathcal{Z}}\geq n_{ \mathcal{X}}$. We denote the
    context space using $ \mathcal{Z}=\Phi( \mathcal{X})$,
    and assume that the set $ \mathcal{Z}$ is
    compact.\label{defn:context}
\end{defn}
We will utilize contexts to enforce available side
information about the unknown dynamics. For example,
consider a robot with a three-dimensoinal state consisting
of its end effector position $(p_x,p_y)\in \mathbb{R}^2$ and
a joint angle $\theta\in[0, 2\pi)$, and we wish to
learn a real-valued function on the state
$f:\mathbb{R}^2\times[0, 2\pi) \to \mathbb{R}$. To enforce
periodicity of $f$ in $\theta$, it suffices to learn a
function $f_\Phi: \mathbb{R}^4 \to \mathbb{R}$ such that
$f(x)=f_\Phi(\Phi(x))$ with
$\Phi(x)=(p_x,p_y,\sin(\theta),\cos(\theta))$. Moreover,
$f_\Phi$ has a Euclidean domain, which admits interpolation
of data. Such nonlinear transformations have been utilized
heavily in machine learning and system identification
communities~\cite{kochenderfer2019algorithms,korda2018linear,kaiser2018sparse}.

\begin{rem}[\textsc{Relaxation of
    Assumption~\ref{assum:nonoise}}]
    In certain problems, the context map $\Phi$ may
    depend on only a subspace of the state space. Then, we can
    relax Assumption~\ref{assum:nonoise} to partial
    state-feedback, where we observe  only the states that
    are relevant for the definition of the context.
\end{rem}

Next, we define an unknown one-step cost function $C$ as a
function of the current context and the applied control.
Formally, we have $C: \mathcal{Z} \times \mathcal{U} \to
\mathbb{R}$. Of special interest is $C$ of the form,
\begin{align}
    C(z_t,u_t)=c(F(\Phi^{-1}(z_t),u_t))\label{eq:one_step_ex}
\end{align}
where $c:\mathcal{X}\to\mathbb{R}$ is a known preference over the next state $x_{t+1}$ at each time instant $t$ and $\Phi^{-1}: \mathcal{Z}\to \mathcal{X}$ is the inverse map of the context map $\Phi$. 

\begin{assum}[\textsc{On-the-fly
    restriction}]\label{assum:onthefly}
    At any time instant $t\in \mathbb{N}$, given the current
    state $x_t$ and the applied control $u_t$, an oracle
    provides the first-order information $(C(z_t,u_t),\nabla
    C(z_t,u_t))$ about $C$ only for the current context
    $z_t$ and control $u_t$.
\end{assum}
The on-the-fly restriction presented in
Assumption~\ref{assum:onthefly} is ubiquitous in realistic
applications. This restriction is also known as bandit
feedback in the contextual optimization
literature~\cite{krause2011contextual}.  Specifically, we
can not evaluate the effect of the control choice $u_t'\neq
u_t$ at the context $z_t$ and time step $t$, once we have
decided to apply the control action $u_t$.  

Assumption~\ref{assum:onthefly} also assumes availability of
exact measurements of $\nabla C$. For simple one-step cost
functions like the relative distance to an object, one can
measure $\nabla C$ (relative velocity) via sensors like the
Doppler radar. However, more complex $C$ may not admit
measurements of $\nabla C$. For dynamics obtained from
continuous-time dynamics, one can also measure $\nabla C$
from perturbations in control~\cite{ornik2019myopic}. Our
numerical examples show that a noisy estimate of the
gradients from finite differences is also sufficient for the
proposed approach.  Our future work will investigate the
impact of noisy oracle on the proposed approach.

For some $N\in \mathbb{N}$, let the finite set of tuples
$\dataSet{N}=\datapointsN$ 
be the initial available data. The data set $\dataSet{N}$  
corresponds to a single
finite-horizon trajectory of length $N$ for the dynamics 
$x_{t+1}=F(x_t,u_t)$ and context map $\Phi$.  
We wish to solve
\eqref{prob:OC_P}, without the explicit knowledge of the
function $C$ (or the nonlinear dynamics $F$):
\begin{align}
    {\mathrm{minimize}}_{u_t\in \mathcal{U}}&\quad
    C(z_t,
    u_t)\qquad\forall t\in \Nint{T},\ t\geq N,\label{prob:OC_P}
\end{align}
under Assumptions~\ref{assum:nonoise}
and~\ref{assum:onthefly}. 

The optimization problem \eqref{prob:OC_P} can be viewed as
an instantiation of nonlinear model predictive control
approach, a standard approach to constrained optimal
control, with a planning horizon of one. The extreme choice
of the planning horizon arises from the fact that the
dynamics $F$ are unknown in the problem of interest.
Existing data-driven model predictive control approaches
that seek to learn the dynamics before control fail to
identify a model or suffer from overfitting in the severe
low-data setting~\cite{kocijan2004gaussian, korda2018linear,
kaiser2018sparse}.  On the other hand, one can use control
augmentation to cast the optimization problem encountered in
a general nonlinear model predictive control formulation  in
the form of \eqref{prob:OC_P}.  However, such an extension
will require access to an accurate look-ahead simulator that
provides $(C(z_t, u_t, u_{t+1}, ...), \nabla C(z_t, u_t,
u_{t+1}, ...))$, which may be hard to obtain in practice.

Equation \eqref{prob:OC_P} can also be viewed in the lens of
reinforcement learning, where $C$ is defined as the
$Q$-function~\cite{bertsekas2019reinforcement}. However,
$Q$-functions are not known \emph{a priori}, and typically
require data-intensive estimations to construct an oracle
that satisfies Assumption~\ref{assum:onthefly}.

We will assume that $C$ is smooth in order to
facilitate interpolation between data points when solving
\eqref{prob:OC_P}.
\begin{assum}[\textsc{$C$ is smooth}]\label{assum:C_Lip}
    $C$ has a Lipschitz continuous gradient, and a known
    Lipschitz gradient constant $L_C$.
\end{assum}
The Lipschitz gradient constant $L_C$ in
Assumption~\ref{assum:C_Lip} restricts the class of dynamics
and one-step cost functions considered, when interpolating
the given data.

We now state the first problem tackled by this paper.
\begin{prob}\label{prob_st:oc_p}
    Using the initial data information $\dataSet{N}$, access
    to full-state feedback, a first-order oracle for $C$, a
    context map $\Phi$, and the knowledge of $L_C$, $
    \mathcal{X}$, and $ \mathcal{U}$, design a control
    algorithm $ \mathscr{A}$ that produces $u_t$ which
    (approximately) solves \eqref{prob:OC_P} at each time
    step $t\in \Nint{T}, t\geq N$. 
\end{prob}
\begin{rem}[\textsc{Problem instance, oracle, \& algorithm}]
    We characterize an instance of the one-step optimal
    control problem by a tuple \ProbInstGen{}.  At each time
    step, an oracle that satisfies
    Assumption~\ref{assum:onthefly} returns the next state
    $x_t$ and the first-order information about the cost
    $(C(z_t,u_t), \nabla C(z_t,u_t))$ based on some fixed
    $\Phi$, $F$, and $C$. On the other hand, an algorithm
    that solves Problem~\ref{prob_st:oc_p} is unaware of the
    dynamics $F$ or the cost $C$.
\end{rem}

A drawback of short horizon-based control as given in
Problem~\ref{prob_st:oc_p} is the lack of foresight. Since
the control actions are chosen greedily to minimize $C$, we
may miss on highly performing action sequences that are
spread across multiple time steps. On the other hand, this
restriction greatly improves the tractability of the
problem~\cite{kocijan2004gaussian,
ornik2019myopic,djeumou2020fly}.  Tractable and scalable
solutions to \eqref{prob:OC_P} can produce effective
data-driven control algorithms.

We also emphasize that \eqref{prob:OC_P} is significantly
more challenging that a general black-box optimization
problem typically considered in the Bayesian optimization
and global optimization
communities~\cite{vinod2020constrained,horst2000introduction,pardalos2010deterministic,kochenderfer2019algorithms}.
In black-box optimization, we seek sequential optimization
algorithms that can converge to the optimal solution of
$\min_{s\in \mathcal{Z}\times \mathcal{U}} J(s)$ for some
unknown function $J: \mathcal{Z}\times\mathcal{U} \to
\mathbb{R}$. 
Due to the dynamics $F$, the optimal action of
\eqref{prob:OC_P} at any time instant depends on the current
context. In contrast, black-box optimization problems have a
time-invariant set of optimal values. Additionally, a
black-box optimization algorithm can query any part of the
feasible space $ \mathcal{Z}\times \mathcal{U}$, while an
algorithm that solves \eqref{prob:OC_P} is restricted to the
set $\{z_t\}\times \mathcal{U}$ at each time step $t$ due to
the on-the-fly restriction.

\subsection{Evaluation of the control algorithm via regret analysis}

We evaluate the performance of the proposed algorithm $
\mathscr{A}$ using the notion of \emph{regret}, inspired by
the notions of contextual regret~\cite{krause2011contextual}
and dynamic regret~\cite{mokhtari2016online}. 
\begin{defn}[\textsc{Regret for one-step control
    algorithm $\mathscr{A}$}]\label{defn:regret}
    At any time step $t\in \mathbb{N}$, the regret incurred by an algorithm $ \mathscr{A}$ that solves Problem~\ref{prob_st:oc_p} is 
    \begin{align}
        \rho_t(\mathscr{A}) = C(z_t, u_t) - C(z_t, u_t^\ast),
        \label{eq:instant_regret}
    \end{align}
    where $u_t$ is the control action selected by $
    \mathscr{A}$ given $z_t=\Phi(x_t)$, and $u_t^\ast$ is the \emph{a priori} unknown optimal solution to \eqref{prob:OC_P}.
\end{defn}
    
Regret accounts for the suboptimality associated with the
application of the control $u_t$ at the time step $t$,
instead of the \emph{a priori unknown, optimal} one-step
control $u_t^\ast$.  By construction, regret is
non-negative.

\begin{defn}[\textsc{Average regret
    for one-step control algorithm $\mathscr{A}$}]
    Given a maximum number of control time steps $T\in \mathbb{N}$ and initial
    trajectory length of $N \in \mathbb{N}$ with $N\leq T$,
    we define the average regret $R_T$
    incurred by the algorithm $ \mathscr{A}$ that solves Problem~\ref{prob_st:oc_p} as 
    \begin{align}
        R_T&=\frac{1}{T-N}\sum_{t=N}^{T}\rho_t\label{eq:R_T_defn}.
    \end{align}
\end{defn}

\begin{prob}\label{prob_st:oc_p_regret}
    Characterize an upper bound on the average regret for
    the proposed solution to Problem~\ref{prob_st:oc_p} with
    $N=0$. 
\end{prob}
\begin{prob}\label{prob_st:oc_p_regret_lb}
    Characterize an algorithm-independent, worst-case lower
    bound on the average regret that every
    one-step control algorithm $\mathscr{A}$ that solves
    Problem~\ref{prob_st:oc_p} must satisfy with $N=0$. 
\end{prob}
We show that the bounds developed to address  Problem~\ref{prob_st:oc_p_regret} and~\ref{prob_st:oc_p_regret_lb} match up to constant factors. Consequently, the proposed algorithm incurs worst-case optimal regret with respect to the number of control time steps $T$. We chose $N=0$ to focus on the performance of the one-step algorithm $\mathscr{A}$ as it interacts with a problem instance of \eqref{prob:OC_P} with no prior knowledge. We can expect the proposed worst-case regret bounds to improve when $N>0$.

\section{Main results}

\subsection{Data-driven control via optimization: Algorithm
details}

\begin{algorithm}[p] \caption{Data-driven control of smooth systems}
    \label{algo:cov_method}
    \begin{algorithmic}[1]
        \STATE \textbf{Input:} Convex \& compact sets $
        \mathcal{X}\subset \mathbb{R}^n$ and
        $\mathcal{U}\subset \mathbb{R}^m$, context map
        $\Phi$, first-order
        oracle for $C$, Lipschitz
        gradient constant $L_C$, time horizon $T$, initial
        data set $\dataSet{N}$ for some $N < T$,
        parameters $\alpha,\beta \in
        \mathbb{R}$ such that $\sqrt{\alpha^2+\beta^2}=1$
        \STATE \textbf{Output:} sequence of control actions
            $\left\{{u^\dagger_t}\right\}_{t={N+1}}^{T}$
        \STATE $t\gets N$ and $x_t\gets x_N$
        \WHILE{$t < T$} 
        \STATE Obtain the current context $z_t \gets
        \Phi(x_t)$
        \STATE Define $\mt(z_t, u_t)=\alpha C_{t-1}^+(z_t,
        u_t) +  \beta C_{t-1}^-(z_t, u_t)$
        \STATE Compute $u_{t}^\dagger$ to apply at the
        current state $x_t$ by solving\label{step:OneStep}
            \begin{align}
                    {\mathrm{minimize}}_{u_t}\ \mt(z_t,
                    u_t)\quad
                    \mathrm{subject\ to}\ u_t\in
                    \mathcal{U}.
                \label{prob:algo_opt_prob}
            \end{align}
            \STATE Obtain the cost information 
            $(C(z_{t},u_t^\dagger),\nabla
            C(z_{t},u_t^\dagger))$ and the next state
            $x_{t+1} \gets F(x_t, u_t^\dagger)$ from the
            oracle 
            \STATE Increment $t$ by one
        \ENDWHILE
    \end{algorithmic}
\end{algorithm}

\begin{figure}[p]
    \centering
    \begin{tikzpicture}
        \tikzstyle{labelarrow}=[thick,
        -{Latex[length=4mm,width=2mm]}]
        \tikzstyle{halfline}=[draw=black,very thick, dotted]
        \tikzstyle{markdot}=[circle,fill, inner sep=2pt]

    \draw[draw=black,very thick,-Latex] (0,0) -- (0:3cm)
        node [xshift=0.25cm] {$\alpha$};
    \draw[draw=black,very thick,-Latex] (0,0) -- (90:3cm)
        node [yshift=-0.25cm, xshift=0.25cm] {$\beta$};

    \draw[color=magenta, very thick, dashed] (0,0) ++
    (0:2cm) arc (0:90:2cm);
    \draw[labelarrow] (30:2cm) ++ (0:1cm) node[text
    width=2cm, anchor=west] {Minimize an
    approximant of $C$} -- (30:2cm);

    \draw[color=blue, very thick, dashed] (0,0) ++
    (135:2cm) arc (135:90:2cm);
    \draw[labelarrow] (115.5:2cm) ++ (135:1cm) node[text
    width=2cm, anchor=south east, xshift=0.25cm,
    yshift=-0.25cm] {Minimize a lower bound on $C$} --
    (115.5:2cm);

    \draw[color=red, very thick, dashed] (0,0) ++
    (0:2cm) arc (0:-45:2cm);
    \draw[labelarrow] (-22.5:2cm) ++ (-45:1cm) node[text
    width=2cm, anchor=north west] {Minimize an upper bound
    on $C$} -- (-22.5:2cm);

    \draw[color=black!100, very thick, dashdotted] (0,0) ++
    (135:2cm) arc (135:315:2cm);
    \draw[labelarrow] (215:2cm) ++ (-135:1.5cm)
    node[text width=2.2cm, anchor=north] {Ignored since
    $\approx$ maximize $C$} -- (215:2cm);

    \draw[halfline] (135:3cm) -- (315:3cm) 
        node[midway, below, xshift=-0.25cm, yshift=0.5cm]
        {\rotatebox{315}{$\alpha+\beta=0$}};

    \draw[halfline] (90:2cm) ++
    (135:1cm) -- (90:2cm) -- (0:2cm) node[midway, below,
    xshift=-0.25cm, yshift=0.5cm]
    {\rotatebox{315}{$\alpha+\beta=1$}};
    \draw[halfline] (0:2cm) ++ (315:1cm) -- (0:2cm);

    \node[markdot] at (315:2cm) {};
    \draw[labelarrow] (315:2cm) ++ (-135:1.1cm) node[text
    width=3.2cm, anchor=north] {Minimize\newline
    uncertainty in $C$} -- (315:2cm);

    \node[markdot] at (135:2cm) {};
    \draw[labelarrow] (135:2cm) ++ (-135:2cm) node[text
    width=2cm, anchor=north] {Maximize uncertainty\linebreak
    in $C$} -- (135:2cm);

    \node[markdot] at (45:2cm) {};
    \draw[labelarrow] (45:2cm) ++ (90:1.5cm) node[text
    width=4cm, anchor=south, xshift=1.25cm] {Minimize the
    pointwise-average of $C^+_{t-1}$ and $C^-_{t-1}$} --
    (45:2cm);

    \node[mark size=5pt,color=black] at (0:2cm)
        {{\color{cyan}\pgfuseplotmark{triangle*}}};

    \node[mark size=5pt,color=black] at (90:2cm)
        {\rotatebox{180}{{\color{cyan}
        \pgfuseplotmark{triangle*}}}};

    \node[draw] at (115:4.5cm) {$\mt=\alpha C_{t-1}^+ +
    \beta C_{t-1}^-$};
\end{tikzpicture}
    \caption{Interpretation of $\alpha,\beta$ parameter
        choices in Algorithm~\ref{algo:cov_method} to define
        $\mt$ \eqref{eq:m_alpha}. We focus on the parameters
        $(\alpha,\beta)$ lying on the unit circle, due to
        the scale-invariance of $\mt$. We mark the tightest
        upper bound $\mt=
        C_{t-1}^+$ with {\color{cyan} $\blacktriangle$}, and the tightest
        lower bound $\mt= C_{t-1}^-$ with
    {\color{cyan} $\blacktriangledown$}. See Table~\ref{tab:characteristics}
for exact ranges of parameter values.}
    \label{fig:characteristics}
\end{figure}
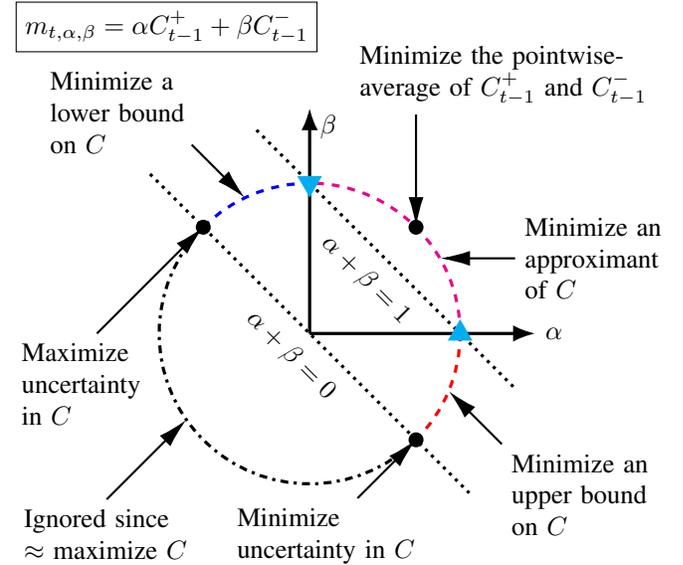

\newcommand{\widthInterpretation}{2.95cm}
\newcommand{\widthAB}{2.2cm}

\begin{table}[p]
   \caption{Interpretation of various choices of
   $\alpha$ and $\beta$ in Algorithm~\ref{algo:cov_method}}
   \label{tab:characteristics}
    \begin{tabular}{cc}
    \hline
    $\alpha, \beta$                   & The choice of
    $u_t^\dagger$ computed using \eqref{prob:algo_opt_prob}\\
    \hline\hline    
                                      & \\[-1ex]
                                      $\beta\leq0,\  \alpha
                                      + \beta \in (0, 1]$ & 
     minimizes an upper bound of $C$ over $\mathcal{U}$
     \\[1ex]
    $\alpha \leq 0,\  \alpha + \beta \in (0, 1]$     
    & minimizes a lower bound of $C$ over $\mathcal{U}$
    \\[1ex]
        $\beta < 0,\ \alpha + \beta=0$
    & has the least uncertainty in $C$ over $\mathcal{U}$
    \\ [2ex]
        $\alpha < 0,\ \alpha + \beta=0$
    &  has the highest uncertainty in $C$ over $\mathcal{U}$
    \\ [2ex]
    $\beta=\alpha > 0$      & \begin{minipage}{5cm}
        \centering
        minimizes the pointwise-average \linebreak of $C_{t-1}^+$ and
        $C_{t-1}^-$ over $\mathcal{U}$ 
    \end{minipage}\\[1ex]
    \hline
   \end{tabular}
\end{table}

We propose Algorithm~\ref{algo:cov_method} to address
Problem~\ref{prob_st:oc_p}. Algorithm~\ref{algo:cov_method} 
uses surrogate optimization to balance the need to
\emph{explore} the action space with the need to
\emph{exploit} the available knowledge to minimize $C$. 
We define a time-dependent surrogate function
$\mt: \mathcal{Z}\times \mathcal{U} \to \mathbb{R}$ for each
time step $t\in \mathbb{N}$,
\begin{align}
    \mt(z,u)&\triangleq \alpha C_{t-1}^+(z,u) + \beta
    C_{t-1}^-(z,u) \label{eq:m_alpha}\\
    &= (\alpha + \beta) C_{t-1}^-(z,u) + \alpha
    \GapC(z,u) \label{eq:m_alpha_gap_Cminus}\\
    &= (\alpha + \beta) C_{t-1}^+(z,u)-\beta \GapC(z,u)
    \label{eq:m_alpha_gap_Cplus},
\end{align}
where $\alpha$
and $\beta \in \mathbb{R}$ are user-specified \noindent
linear combination parameters, and $$\GapC(z,u)\triangleq
C_{t-1}^+(z,u) - C_{t-1}^-(z,u) \geq 0$$ for every $(z,u)\in
\mathcal{Z}\times \mathcal{U}$.
Algorithm~\ref{algo:cov_method} iteratively minimizes $\mt$
given $z_t$ to compute the one-step control action
$u_t^\dagger$ using only
the data known at time instant $t$. The computed action $u_t^\dagger$
approximates the optimal action $u_t^\ast$
used to define $\rho_t$ in \eqref{eq:instant_regret}.

Figure~\ref{fig:characteristics} and
Table~\ref{tab:characteristics} illustrates the implications
various choices of $\alpha$ and $\beta$. 
By construction, the control actions selected by
Algorithm~\ref{algo:cov_method} corresponding to
any parameter choice $(\alpha,\beta)$ is the same as the
$\left(\frac{\alpha}{\sqrt{\alpha^2 +
\beta^2}},\frac{\beta}{\sqrt{\alpha^2 + \beta^2}}\right)$.
Therefore, it suffices to focus on the parameter choices on
the unit circle as shown in
Figure~\ref{fig:characteristics}.  Since $\GapC\geq 0$,
\begin{align}
    \frac{\mt}{\alpha + \beta}&= \begin{cases}
        \begin{array}{ll}
            C_t^- - \frac{|\alpha|}{\alpha + \beta}
            \GapC \quad\leq C_t^-,& \alpha \leq 0\\
            C_t^+ + \frac{|\beta|}{\alpha + \beta}
            \GapC\quad\geq C_t^+,& \beta \leq 0\\
        \end{array}
    \end{cases}\label{eq:mt_scaled}.
\end{align}
for $\alpha + \beta > 0$, thanks to
\eqref{eq:m_alpha_gap_Cminus} and
\eqref{eq:m_alpha_gap_Cplus}. Thus,  \eqref{eq:mt_scaled}
justifies the first two rows of
Table~\ref{tab:characteristics},
\eqref{eq:m_alpha_gap_Cminus} and
\eqref{eq:m_alpha_gap_Cplus} justifies the third and fourth
rows, and the fifth row follows from \eqref{eq:m_alpha}.
Table~\ref{tab:characteristics} shows that, by varing the
linear combination parameters $\alpha$ and $\beta$, $\mt$
covers a variety of data-driven surrogate functions.  For
example, minimizing lower bounds of the unknown objective
function have been considered under \emph{the principle of
optimism under
uncertainty}~\cite{auer2002using,bogunovic2020corruption}.

We ignore parameter choices with $\alpha + \beta < 0$, since
it conflicts the goal of the problem formulation
\eqref{prob:OC_P}. Specifically, we observe that  controller synthesis via
minimization of $\mt$ with parameters $\alpha$ and $\beta$
with $\alpha + \beta < 0$ is equivalent to the
\emph{maximization} of an upper bound, lower bound, or an
approximant on $C$ for parameters $\alpha'=-\alpha$ and
$\beta'=-\beta$ with $\alpha'+\beta'>0$.

\subsection{Convexity-based properties of Algorithm~\ref{algo:cov_method}}

We study the
convexity properties of the surrogate function $\mt$ in  
Theorem~\ref{thm:mt_cvx}. Recall that difference-of-convex
functions are functions that can be expressed as a
difference of two convex functions~\cite{an2005dc}.
\begin{thm}[\textsc{Convexity-based properties of
    $\mt$}]\label{thm:mt_cvx} For any time step $t\in
    \mathbb{N}$, the function $\mt$ is:
    {\renewcommand{\theenumi}{\roman{enumi}}
    \begin{enumerate}
        \item a difference-of-convex function over 
            $\mathcal{Z}\times \mathcal{U}$ for any
            $\alpha,\beta \in \mathbb{R}$,
        \item a piecewise convex quadratic function over 
            $\mathcal{Z}\times \mathcal{U}$ for $\alpha >
            0 \geq \beta$, 
        \item a piecewise concave quadratic function over
            $\mathcal{Z}\times \mathcal{U}$ for $\beta >
            0 \geq \alpha$,
            and 
        \item a piecewise linear function over 
            $\mathcal{Z}\times \mathcal{U}$ for $\alpha
            =\beta$.
    \end{enumerate}}
\end{thm}
\begin{proof}
    For any $q,s\in \mathcal{Z}\times \mathcal{U}$, recall
    that the quadratic function $Q(s)=\frac{s \cdot s}{2}$
    satisfies the relation
    \begin{align}
        Q(s) &=\frac{s \cdot s}{2}= \frac{1}{2}\|s - q\|^2 
        + \ell(s;Q, q)\label{eq:quad}.
    \end{align}
    Also, define $q_i=(z_i,u_i)\in \mathcal{Z}\times
    \mathcal{U}$ for each $i\in\Nint{t-1}$ as the past
    context-control pairs known from the data available at
    the time step $t$, $\{\datapointsi\}_{i\in\Nint{t-1}}$.

    On expanding \eqref{eq:m_alpha} at $t$, we
    obtain
    \begin{align}
        \mt(s) &=\alpha \min_{i\in \Nint{t-1}}\left({
        \ell(s;C,q_i) + \frac{L_C}{2} \|s - q_i\|^2}\right)
        \nonumber \\
          &\quad + \beta\max_{j\in\Nint{t-1}}\left({
          \ell(s;C,q_i) - \frac{L_C}{2} \|s -
  q_i\|^2}\right) \nonumber \\
          &=\underbrace{\frac{L_C}{2}(s \cdot s)\left(
          \alpha - \beta\right)}_{\triangleq f_1(s)} \nonumber
          \\
          &\quad+ \underbrace{\alpha\min_{i\in \Nint{t-1}}\left({
          \ell(s;C,q_i) -
          L_C\ell(s;Q,q_i)}\right)}_{\triangleq f_2(s)}
          \nonumber \\
          &\quad + \underbrace{\beta\max_{j\in
                  \Nint{t-1}}\left({\ell(s;C,q_j) +
L_C\ell(s;Q,q_j)}\right)}_{\triangleq f_3(s)} \label{eq:dc_decompose}.
    \end{align}
    Here, $f_1$ is a quadratic function for
    $\alpha\neq\beta$, while $f_2$ and $f_3$ are
    piecewise linear functions for $\alpha\neq0$ or
    $\beta\neq 0$ respectively.

    \begin{table}
        \caption{Convexity properties of $f_1$, $f_2$, and
        $f_3$ for $\mt=f_1+f_2+f_3$.}
        \label{tab:mt_decompose}
        \begin{tabular}{|c|c|c|c|}
            \hline
            {\begin{minipage}{1.25cm}\centering $\alpha,\beta$\end{minipage}} 
            & Quadratic $f_1$ & Piecewise linear $f_2$ 
            & Piecewise linear $f_3$\\\hline\hline
            $\alpha > 0 \geq \beta$ &
            Convex & \multicolumn{2}{c|}{Concave}\\\hline
            $\alpha > \beta \geq 0$ & Convex &
            \multirow{3}{*}{Concave}  &
            \multirow{3}{*}{Convex}\\\cline{1-2}
            $\alpha=\beta \geq 0$ & Zero & & \\\cline{1-2}
            $\beta> \alpha \geq 0$ & Concave &  & \\\hline
            $\beta> 0 \geq \alpha$ &
            Concave & \multicolumn{2}{c|}{Convex}\\\hline
        \end{tabular}
    \end{table}

    Table~\ref{tab:mt_decompose} summarizes the
    convexity/concavity of the functions $f_1,f_2,$ and
    $f_3$ for various ranges where either $\alpha$ or
    $\beta$ or both are nonnegative. In all of these cases,
    we notice that $\mt$ is a sum of a convex and a concave
    function. We know that $\mt$ is a sum of a convex and a
    concave function when $\alpha$ and $\beta$ are negative,
    since $\mt=-\mtprime$ for $\alpha'=-\alpha>0$
    and $\beta'=-\beta>0$. Since a
    sum of a convex and a concave function is also a
    difference-of-convex function and difference-of-convex
    functions are closed under negation, we have the proof
    of i).
    
    For $\alpha > 0\geq \beta$ and $\beta > 0\geq
    \alpha$, $f_2 + f_3$ can be expressed as a minimum and a
    maximum of $t^2$ linear functions respectively.
    Consequently, we see that $\mt$ is a minimum and a
    maximum of convex and concave quadratics in these cases
    (see first and last row of
    Table~\ref{tab:mt_decompose}). This completes the proof
    for ii) and iii).

    Finally, iv) follows from the middle row
    in Table~\ref{tab:mt_decompose}.
\end{proof}

\begin{corr}[\textsc{Algorithm~\ref{algo:cov_method}
    requires convex or non-convex
optimization based on parameter choices}]\label{corr:algo_opt_prob}
    For any time step $t\in \mathbb{N}$, the optimization
    problem \eqref{prob:algo_opt_prob} can be solved via
    {\renewcommand{\theenumi}{\roman{enumi}}
    \begin{enumerate}
        \item the finite minimum of at most $t^2$ convex quadratic
            programs for $\alpha > 0 \geq \beta$, 
        \item the finite minimum of at most $t$ convex quadratic
            programs for $\alpha > \beta \geq 0$,
        \item the finite minimum of at most $t$ linear programs for
            $\alpha = \beta$,
        \item a single concave quadratic minimization
            program for $\beta > \alpha \geq 0$, and
        \item the finite minimum of at most $t$ concave
            quadratic minimization programs for $\beta >
            0 \geq \alpha$.
    \end{enumerate}}
    \begin{proof}
        The proof of i) relies on the observation that $f_1$
        is a convex quadratic function and $(f_2+f_3)$ is a
        concave piecewise linear function, with at most
        $t^2$ pieces, when $\alpha > 0 \geq \beta$. In other
        words, $\mt$ is a minimum of at most $t^2$ convex
        quadratic functions~\cite{an2005dc}. Consequently,
        we can solve the optimization problem
        \eqref{prob:algo_opt_prob} by considering each one
        of the quadratic function separately, solve the
        corresponding convex quadratic programs, and then
        compute the finite minimum of these programs.

        The proof of ii) relies on the observation that
        $f_2$ is a finite minimum of at most $t$ linear
        functions, and $f_3$ can be formulated as an
        optimization over a polytope using the epigraph
        reformulation~\cite[Sec. 4.3.1]{BoydConvex2004}.
        Thus, for each of the linear function in $f_2$, we
        arrive at a convex quadratic program.  Similarly to
        i), we can compute the optimal solution of
        \eqref{prob:algo_opt_prob} via the finite minimum of
        $t$ convex quadratic programs.

        The proof of iii) follows similarly to ii), with the
        only difference that $f_1=0$. Consequently, we can
        compute the optimal solution of
        \eqref{prob:algo_opt_prob} via the finite minimum of
        at most $t$ linear programs.

        Using the epigraph formulation~\cite[Sec.
        4.3.1]{BoydConvex2004} to handle minimization of the
        resulting convex piecewise linear function
        $f_2+f_3$, we minimize a concave quadratic function
        over a polytope for iv).

        The proof of v) follows similarly to ii), with the
        difference that $f_1$ is now a concave quadratic
        function. Consequently, we have to consider the
        finite minimum of solutions of $t$ concave quadratic
        minimization problems.
    \end{proof}
\end{corr}

\begin{figure}[h]
    \centering
    \vspace*{-0.5em}
    \begin{tikzpicture}
        \tikzstyle{labelarrow}=[thick,
        -{Latex[length=4mm,width=2mm]}]
        \tikzstyle{halfline}=[draw=black,very thick, dotted]
        \tikzstyle{markdot}=[circle,fill, inner sep=2pt]

    \fill[fill=black!30] (0:0cm) -- (0:2cm) arc (0:45:2cm)
    -- cycle;
    \fill[fill=black!30] (0:0cm) -- (45:2cm) arc (45:90:2cm)
    -- cycle;
    \fill[fill=black!10] (0:0cm) -- (90:2cm) arc (90:135:2cm)
    -- cycle;
    \fill[fill=black!10] (0:0cm) -- (135:2cm) arc
    (135:180:2cm) -- cycle;

    \draw[draw=black,very thick,-Latex] (0,0) -- (45:3.5cm)
        node [xshift=0.25cm] {$\alpha$};
    \draw[draw=black,very thick,-Latex] (0,0) -- (135:3.5cm)
        node [xshift=-0.25cm] {$\beta$};
    \draw[draw=black,very thick, dotted] (0,0) -- (90:4cm);

    \draw[color=blue, very thick, dashed] (0,0) ++
    (180:2cm) arc (180:135:2cm);

    \draw[color=red, very thick, dashed] (0,0) ++
    (45:2cm) arc (45:0:2cm);

    \draw[color=magenta, very thick, dashed] (0,0) ++
    (45:2cm) arc (45:135:2cm);

    \draw[halfline] (180:4.25cm) -- (0:4.25cm) 
        node[midway, below, xshift=-0.25cm]
        {{$\alpha+\beta=0$}};


    \node[mark size=5pt,color=black] at (45:2cm)
        {{\color{cyan}\pgfuseplotmark{triangle*}}};

    \node[mark size=5pt,color=black] at (135:2cm)
        {\rotatebox{180}{{\color{cyan}
        \pgfuseplotmark{triangle*}}}};

    \node[markdot] at (0:2cm) {};
    \node[draw, very thin, xshift=1cm, fill=white] at (0:2cm) {$\GapC$};
    \node[markdot] at (90:2cm) {};
    \node[draw, yshift=0.8cm, fill=white] at (90:2cm) {$\frac{C_t^- + C_t^+}{2}$};
    \node[markdot] at (180:2cm) {};
    \node[draw, xshift=-1cm, fill=white] at (180:2cm)
    {$-\GapC$};
    \node[draw, xshift=-1cm] at (135:2cm) {$C_t^-$};
    \node[draw, xshift=1cm] at (45:2cm) {$C_t^+$};

    \node[xshift=0.25cm, yshift=0.15cm] at (22.5:1cm) {i)};
    \node[xshift=0.15cm, yshift=0.25cm] at (67.5:1cm) {ii)};
    \node[xshift=-0.175cm, yshift=0.2cm] at (102.5:1cm)
    {iv)};
    \node[xshift=-0.25cm] at (147.5:1cm) {v)};
    \node[yshift=-0.5cm, xshift=0.25cm] at (90:4cm) {iii)};

    \draw  [|-Latex, postaction={decorate, decoration={text
        align={center},raise={1mm},text along path,
text={|\small|Concave quadratic minimization}}}] (0,0) ++
(180:4cm) arc (180:90:4cm);
    \draw  [|-|, postaction={decorate, decoration={text
        align={center},raise={1mm},text along path,
text={|\small|Convex quadratic/linear minimization}}}] (0,0) ++
(90:4cm) arc (90:0:4cm);

    \end{tikzpicture}
    \caption{Illustration of the convexity/non-convexity of
        \eqref{prob:algo_opt_prob} in 
        Algorithm~\ref{algo:cov_method} under different
        parameter choices via
        Corollary~\ref{corr:algo_opt_prob}. For parameter
        choices in the
        darker-shaded regions $(\alpha
        \geq \beta, \alpha+\beta \geq 0)$, we can solve
        \eqref{prob:algo_opt_prob} using a collection of convex
        quadratic
        programs, while in the lighter shaded regions $(\alpha <
        \beta, \alpha+\beta \geq 0)$, we need to solve concave quadratic
minimization programs. The special case $\alpha=\beta$ only requires 
linear programming, see~\cite{vinod2020convexified}.}
    \label{fig:cvx_ncvx}
\end{figure}
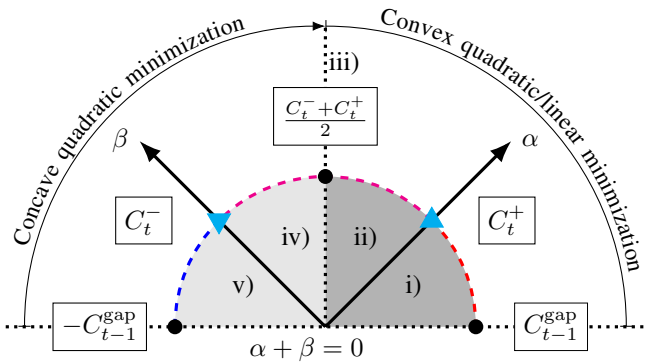

We illustrate the implications of
Corollary~\ref{corr:algo_opt_prob} in
Figure~\ref{fig:cvx_ncvx}.  The optimization problem
\eqref{prob:algo_opt_prob} may be solved via convex
optimization in cases i), ii), and iii), i.e.,  $\alpha\geq
\beta$ and $\alpha + \beta \geq 0$. In the remaining cases
iv) and v), i.e.,  $\alpha< \beta$ and $\alpha + \beta \geq
0$, we need to solve concave quadratic minimization
problems, which are harder to solve in practice. However,
commercial solvers like \texttt{GUROBI} can now handle these
problems using spatial branching~\cite{gurobi}. 

\begin{rem}[\textsc{Implementation of
        Algorithm~\ref{algo:cov_method}}]
    Corollary~\ref{corr:algo_opt_prob} shows that
    Algorithm~\ref{algo:cov_method} can be implemented using
    an off-the-shelf solver like
    \texttt{GUROBI}~\cite{gurobi} for every
    relevant parameter choice of $\alpha+\beta \geq 0$.
\end{rem}

\begin{rem}[\textsc{Relationship with our prior work}]
    In~\cite{vinod2020convexified}, we proposed \texttt{C2Opt} that
    approximately solves \eqref{prob:OC_P} via convexified
    contextual optimization. Algorithm~\ref{algo:cov_method}
    simplifies to \texttt{C2Opt} when $\alpha=\beta > 0$,
    which corresponds to case iii) in
    Corollary~\ref{corr:algo_opt_prob}.
\end{rem}

\subsection{Imposing available high-level information on $C$}

In some cases, we may additionally know that $C$ is convex.
For example, consider an instantiation of \eqref{prob:OC_P}
with an identity context map, control-affine unknown
dynamics $x_{t+1} = F(x_t,u_t)=F_1(x_t)+F_2(x_t)u_t$, and
$C(x_t,u_t)=c(x_{t+1})$ for some convex function $c:
\mathcal{X}\to \mathbb{R}$. It is easy to see that $C$ is
convex over the control action $u_t$ for every fixed current
state $x_t$ at every time step $t$.  In such cases, we can
use a tighter, convex and piecewise-linear, data-driven
minorant $C_{t,\mathrm{convex}}^-: \mathcal{Z}\times
\mathcal{U}\to \mathbb{R}$ instead of $C_t^-$,  
\begin{align}
    C_{t,\mathrm{convex}}^-(s)&\triangleq\max_{i\in\Nint{t}}\ C(q_i)
    + \nabla
    C(q_i) \cdot (s-q_i)\label{eq:convex_C_minus},
\end{align}
with $C_t^- \leq C_{t,\mathrm{convex}}^- \leq C$.
\begin{rem}
    The data-driven, one-step optimal control
    \eqref{prob:algo_opt_prob} in
    Algorithm~\ref{algo:cov_method} is convex for any
    $\alpha \geq 0$, which corresponds to cases i), ii),
    iii), and iv) in Corollary~\ref{corr:algo_opt_prob}.
\end{rem}

We can also tighten the data-driven minorant and
majorants of $C$, when $C$ is additionally known to be
monotone, bounded from above or below, or when $\nabla
C$ is known to lie in a convex set. We refer the reader
to~\cite[Tab. I]{vinod2020convexified} for the suitable
modifications to the computation of $C^+$  and $C^-$ to
incorporate such side information.

\begin{figure}
    \centering
    \includegraphics[width=1\linewidth]{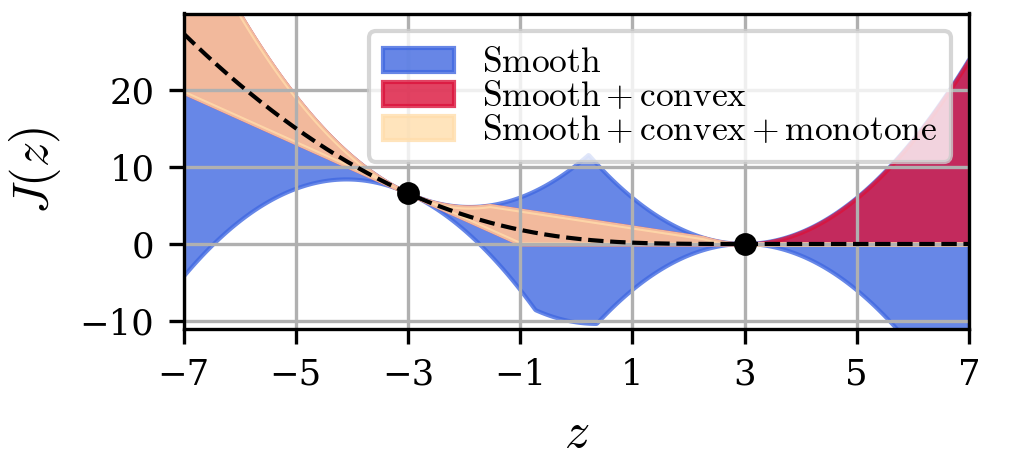}
    \caption{Approximation bounds
        $J^\pm(\cdot)$ for
    $J(z)=-\log(\mathrm{NormalCDF}(z))$ (dotted line)
computed from two data points (black dots) with varying
levels of side information; Lipshitz gradient constant $L_J=3$}
    \label{fig:norm_cdf}
\end{figure}
Figure~\ref{fig:norm_cdf} shows the data-driven bounds
$J^\pm(\cdot)$ for $J(z)=-\log(\mathrm{NormalCDF}(z))$ using
various types of side information from only two data points.
In this case, $J(z)$ is known to be monotone decreasing and
convex~\cite[Ex. 3.39]{BoydConvex2004}. As expected, imposing additional side
information significantly improves the tightness of the
bounds $J^\pm(\cdot)$.  Remarkably, we recover the unknown
function $J(z)$ for $z\geq 3$ with $J^+(z)=J^-(z)=J(z)=0$,
when we impose the monotone decreasing and convex property
of $J(\cdot)$.

\section{Regret analysis for Algorithm~\ref{algo:cov_method}}

We now turn our attention to
Problems~\ref{prob_st:oc_p_regret}
and~\ref{prob_st:oc_p_regret_lb} --- the regret analysis of
Algorithm~\ref{algo:cov_method}. First, we show that
Algorithm~\ref{algo:cov_method} enjoys an upper bound on the
regret when $\alpha \leq 0$ and $\alpha + \beta >
0$. Next, we characterize an upper
bound on the average regret of
Algorithm~\ref{algo:cov_method} with $\alpha=0$ and $\beta
> 0$. Finally, we show that the proposed upper bound
matches an algorithm-independent lower bound up to a
constant factor demonstrating the worst-case optimality of
Algorithm~\ref{algo:cov_method} with respect to the average regret. 

\subsection{Upper bound on the regret for
Algorithm~\ref{algo:cov_method}}

\begin{prop}[\textsc{Upper bound on
    regret}]\label{prop:dyn_reg} 
    For any $t\in \mathbb{N}\setminus\Nint{N-1}$, the
    regret of Algorithm~\ref{algo:cov_method} at time
    step $t$ is bounded from above for $\alpha \leq 0$ and
    $\alpha + \beta > 0$, 
    \begin{align}
        \rho_t\leq L_C\left(1 + \frac{|\alpha|}{\alpha +
        \beta}\right) \min_{i\in\Nint{t-1}}\|s_t -
        q_i\|^2\label{eq:Dyn_Reg_UB},
    \end{align}
    where $s_t=(z_t, u_t)\in \mathcal{Z}\times
    \mathcal{U}$ and $q_i$, $i\in \Nint{t-1}$ are the past
    context-control action pairs. 
\end{prop}
\begin{proof}
    Since $u_t$ minimizes $\mt$, $u_t$ minimizes         
    $\frac{1}{\alpha + \beta}\mt$ for
    $\alpha+\beta > 0$. Since $\alpha\leq 0$, we have,
    by using \eqref{eq:mt_scaled},
    \begin{align}
        \frac{\mt(z_t, u_t)}{\alpha + \beta} &\leq
        \frac{\mt(z_t, u_t^\ast)}{\alpha + \beta} \nonumber \\
                                                     &=
        C_t^-(z_t, u_t^\ast) - \frac{|\alpha|}{\alpha +
        \beta}\GapC(z_t, u_t^\ast) \nonumber \\
        &\leq C_t^-(z_t, u_t^\ast) \leq C(z_t,
        u_t^\ast),\label{eq:Dyn_Reg_LB}
    \end{align}
    at every time step $t\in \mathbb{N}\setminus\Nint{N-1}$.
    Recall that $\GapC$ is a
    non-negative function. 
    
    From \eqref{eq:mt_scaled} and \eqref{eq:Dyn_Reg_LB}, the
    regret at $t$ is bounded from above,
    \begin{align}
        \rho_t &= C(z_t, u_t) - C(z_t, u_t^\ast) \nonumber \\
               &\leq C(z_t, u_t) - \frac{\mt(z_t,
               u_t)}{\alpha + \beta} \nonumber \\
               &\leq C(z_t, u_t) - C_t^-(z_t,
               u_t) + \frac{|\alpha|}{\alpha +
               \beta}\GapC(z_t, u_t) \nonumber \\
               &\leq \left(1 + \frac{|\alpha|}{\alpha +
               \beta}\right)\GapC(z_t, u_t).
    \end{align}
    By Lemma~\ref{lem:approx}, $\GapC(z_t, u_t) \leq
    L_C\min_{i\in\Nint{t-1}}\|s_t - q_i\|^2$. Thus, $\rho_t
    \leq L_C\left(1 + \frac{|\alpha|}{\alpha + \beta}\right)
    \min_{i\in\Nint{t-1}}\|s_t - q_i\|^2$, as desired.
\end{proof}
We next recall the volume counting lemma, which we prove for
the sake of completeness in Appendix~\ref{app:lem_pidgeonhole}. 
\begin{lem}[\textsc{Volume counting lemma}]\label{lem:pidgeonhole}
    For any $\epsilon > 0$, any compact set $
    \mathcal{S} \subset \mathbb{R}^d$, choose
    \begin{align}
        T\geq \left\lceil{
           \left( 
               \diam(\mathcal{S})\sqrt{\frac{d}{\epsilon}}
           \right)}^d
        \right\rceil + 1.
    \end{align}
    For any finite collection of distinct points
    $q_i\in \mathcal{S}$ for every $i\in \Nint{T}$, there
    exists $t\in \Nint{T-1}$ such that
    $\min_{i\in\Nint{t}}\|q_{t+1} - q_i\|^2\leq \epsilon$.
\end{lem}
\begin{corr}\label{corr:vol_bd}
    For any $\delta > 0$, $\lvert \{ t \in [T] : \rho_t \geq
    \delta \}\rvert \leq M_0 \delta^{-d/2}$ with $d=n_{
\mathcal{Z}} + n_{ \mathcal{U}}$, where the
    constant $M_0$ is determined by $L_C$, $\alpha$,
    $\beta$, and the diameters of the sets $\mathcal{Z}$ and
    $\mathcal{U}$.
\end{corr}
\begin{proof}
    From Proposition~\ref{prop:dyn_reg}, $\rho_t \leq
    M_0'\min_{i\in\Nint{t-1}}\|q_t - q_i\|^2$ for every
    $t\in \mathbb{N}\setminus\Nint{N-1}$, where
    $M_0'=L_C\left(1 + \frac{|\alpha|}{\alpha +
    \beta}\right)$. Consequently, for any $\delta > 0$,
    \begin{align}
        &\lvert \{ t \in [T] : \rho_t \geq \delta \}\rvert
        \nonumber \\
        &\leq \lvert \{ t \in [T] :
        M_0'\min_{i\in\Nint{t-1}}\|q_t - q_i\|^2 \geq
    \delta \}\rvert \nonumber \\
        &= \lvert \{ t \in [T] : \forall i\in\Nint{t-1}, 
        \|q_t - q_i\|^2 \geq \delta/M_0' \}\rvert \nonumber \\
        &= T - \lvert \{ t \in [T] : \exists i\in\Nint{t-1}, 
        \|q_t - q_i\|^2 < \delta/M_0' \}\rvert \nonumber \\
        &\leq T - \left(T -
        {\left(\diam(\mathcal{Z} \times \mathcal{U})\sqrt{\frac{d}{\delta/M_0'}}
    \right)}^d\right)=M_0 \delta^\frac{-d}{2}, \nonumber
    \end{align}
    with
    $M_0={\left({\diam(\mathcal{Z} \times \mathcal{U})\sqrt{M_0'd}}\right)}^{d}$.
    We have the inequality in the last step from
    Lemma~\ref{lem:pidgeonhole}.
\end{proof}
Corollary~\ref{corr:vol_bd} is crucial for addressing
Problem~\ref{prob_st:oc_p_regret}.

\subsection{Upper bound on the average regret}

We obtain an upper bound on the average regret
under the following assumption on $C$.
\begin{assum}[\textsc{$C$ has a norm-bounded gradient}]\label{assum:norm_bound}
            Let $\|\nabla C(z)\|\leq \GCmax$ for every $z
            \in \mathcal{Z}\times \mathcal{U}$ for some
            known $\GCmax > 0$.
\end{assum}
Assumption~\ref{assum:norm_bound} is mild due to the
smoothness assumption on $C$ and compactness of the sets $
\mathcal{Z}$ and $ \mathcal{U}$.  Specifically, in the
scenario where a gradient bound $ \GCmax$ is not known
\emph{a priori}, $ \GCmax = L_C \diam( \mathcal{Z}\times
\mathcal{U})$ satisfies Assumption~\ref{assum:norm_bound} by
the Mean Value Theorem, provided $C$ has at least one
stationary point over $ \mathcal{Z}\times \mathcal{U}$.

Next, we state the following auxiliary result to demonstrate
an upper bound on the average regret.
We provide the proof of Lemma~\ref{lem:ibp2} in
Appendix~\ref{app:lem_ibp2}.
\begin{lem}
    \label{lem:ibp2}
Let $\mu > 1$ and $A>0$. Suppose $\{n_k\}_{k \geq 1}$ is a
sequence of non-negative  real numbers with $\sum_{k=1}^\infty
n_k = T$, and $n_k \leq A e^{\mu k}$ for all $k$. Then, 
\[
\sum\nolimits_{k=1}^\infty n_k e^{-k} \leq \mu (\mu-1)^{-1}
A^{1/\mu} T^{1-1/\mu}.
\]
\end{lem}
\begin{thm}[\textsc{Upper bound on average
    regret}]\label{thm:sharp_ubd}
    Suppose $n_{\mathcal{Z}} +
    n_{\mathcal{U}} \geq 3$ and $N=0$.
    Algorithm~\ref{algo:cov_method} with
    $\alpha=0$ and $\beta>0$ incurs an average
    regret $R_T \leq \widehat{M}T^{\frac{-2}{n_{\mathcal{Z}} +
    n_{\mathcal{U}}}}$, where $\widehat{M}$ is a
    positive constant that depends on $L_C$, $\GCmax$,
    $\alpha$, $\beta$, and the diameters and dimensions 
    of the sets $\mathcal{Z}$ and $\mathcal{U}$.
\end{thm}
\begin{proof}
    Define $\Lambda = \GCmax
    \diam(\mathcal{U})$. By the Mean Value Theorem and
    Assumption~\ref{assum:C_Lip}, we
    have $\rho_t \leq \Lambda$ for all $t$. Thus, we can write
    \[
    \begin{aligned}
    \sum_{t=1}^T \rho_t &\leq \sum_{k=1}^\infty e^{-k+1}
    \Lambda \underbrace{\lvert \{ t \in [T] : \rho_t \in (e^{-k}
    \Lambda, e^{-k+1} \Lambda ]\}\rvert}_{\triangleq n_k} \\
    &= e \Lambda \sum_{k=1}^\infty n_k e^{-k}.
    \end{aligned}
    \]
    Define $d=n_{\mathcal{Z}} + n_{\mathcal{U}}$.  From
    Corollary~\ref{corr:vol_bd},
    \[
    \begin{aligned}
    n_k \leq \lvert{\{ t \in [T] : \rho_t > e^{-k}
    \Lambda\}}\rvert &
    \leq M_0 (e^{-k} \Lambda)^\frac{-d}{2}= M_0
    \Lambda^\frac{-d}{2} e^\frac{kd}{2},
    \end{aligned}
    \]
    where the
    constant 
    $M_0$ is determined by $L_C$ and the
    diameters of the sets $\mathcal{Z}$ and $\mathcal{U}$.
    Note that $\sum_{k=1}^\infty n_k = T$. Consequently, we
    apply Lemma~\ref{lem:ibp2} to the sequence $\{n_k\}$,
    with $\mu = d/2 > 1$ and $A = M_0 \Lambda^{-d/2} >
    0$, 
    \begin{align}
        R_T = \frac{1}{T}\sum_{t=1}^T \rho_t &\leq
        \frac{e\Lambda}{T}
    \sum_{k=1}^\infty n_k e^{-k} \leq
    \widehat{M}T^{-2/d}\label{eq:regret_ub_last_step}
    \end{align}
    for a constant $\widehat{M}$, as claimed.
\end{proof}
\begin{corr}[\textsc{Sublinear cumulative regret}]
    Suppose $n_{\mathcal{Z}} + n_{\mathcal{U}} \geq 3$ and
    $N=0$. Algorithm~\ref{algo:cov_method} with $\alpha=0$
    and $\beta>0$ incurs sublinear cumulative regret,
    $\sum_{t=1}^T\rho_t \leq \widehat{M}T^{1 -
        \frac{2}{n_{\mathcal{Z}} + n_{\mathcal{U}}}}$.
\end{corr}
\begin{rem}
    Theorem~\ref{thm:sharp_ubd} requires $n_{\mathcal{Z}} +
    n_{\mathcal{U}} > 2$ in order to utilize
    Lemma~\ref{lem:ibp2} with $\mu=(n_{\mathcal{Z}} +
    n_{\mathcal{U}})/2>1$. For $n_{\mathcal{Z}} +
    n_{\mathcal{U}} = 2$, the cumulative regret of
    Algorithm~\ref{algo:cov_method}  with $\alpha=0$
    and $\beta>0$ is also sublinear.
    See Appendix~\ref{app:diff_lem} for more details.
\end{rem}

\subsection{Algorithm-independent lower bound on the average
regret when solving
Problem~\ref{prob_st:oc_p}}

Next, we demonstrate a lower bound on the average regret for any control algorithm $\mathscr{A}$ that solves
Problem~\ref{prob_st:oc_p} with no initial data ($N=0$) and
identity context map $\IdentityMap$
($z_t=\IdentityMap(x_t)=x_t$ and $n_{ \mathcal{Z}}=n_{
\mathcal{X}}$). Specifically, we show that for
every control algorithm $\mathscr{A}$ there exists a
problem instance \ProbInstZero{} that suffers from an average regret with a budget-dependent lower bound. We call such
problem instances \emph{resisting}, since they ensure that
some cost function $C$ and dynamics
$F$ exist that satisfy 
Assumptions~\ref{assum:nonoise}--~\ref{assum:norm_bound},
without explicitly constructing them.

\begin{defn}[\textsc{$(\xi,T)$-resisting property
    to control algorithm $ \mathscr{A}$}]\label{defn:xi}
    Given $\xi>0$, a budget $T$, and an initial state
    $x_0\in \mathcal{X}$. A problem instance
    \ProbInst{\IdentityMap}{\emptyset} is said to be
    $(\xi,T)$-resisting to control algorithm $ \mathscr{A}$, provided \\
    1) 
    the trajectory ${\{x_t\}}_{t\in \Nint{T}}$ resulting
    from the application of the control $u_t$ chosen by the
    algorithm $ \mathscr{A}$ at each time step $t\in
    \mathbb{N}$ satisfies $\|x_{t+1} - x_i\| > \xi$ for
    every $i\in \Nint{t}$, and\\ 
    2) the oracle returns $C(z_t, u_t)=0$ and $\nabla C(z_t, u_t)=0$ at each time step $t\in \Nint{T}$.
\end{defn}

\begin{thm}[\textsc{Lower bound on average
        contextual 
    regret}]\label{thm:sharp_lbd}
    Given $T\in
    \mathbb{N}$ with $T>0$, define\linebreak
    $\xi=
    \diam(\mathcal{X})/(\sqrt{n_{\mathcal{X}}}T^\frac{1}{n_{\mathcal{X}}})$,
    and let $ \mathcal{X}$ and $ \mathcal{U}$
    have strictly positive diameters.  For any control
    algorithm $\mathscr{A}$ that solves
    Problem~\ref{prob_st:oc_p} under 
    Assumptions~\ref{assum:nonoise}--~\ref{assum:norm_bound},\\
    1) there exists a problem instance \ProbInstZero{} that
    is $(\xi,T)$-resisting for the control
    algorithm $\mathscr{A}$,\\
    2) the average regret incurred by $
    \mathscr{A}$ on the $(\xi,T)$-resisting problem instance
    is $R_T \geq
    \nu{T^{\frac{-2}{n_{\mathcal{X}} + n_{\mathcal{U}}}}},$
    where the constant $\nu$ depends only on $L_C$,
    $n_{\mathcal{X}}$, and $n_{\mathcal{U}}$.
\end{thm}

We provide the proof of Theorem~\ref{thm:sharp_lbd} in
Appendix~\ref{app:thm_sharp_lbd}. The key insight is to show
that an
algorithm-dependent problem instance always exist
for which the oracle provides no useful information
regarding the unknown cost function irrespective of the control input 
selected by the algorithm $ \mathscr{A}$
($(\xi,T)$-resisting property), but causes the algorithm to
suffer from non-trivial regret at each time step. 

Surprisingly, the worst-case average regret incurred by any
one-step algorithm that solves Problem~\ref{prob_st:oc_p}
depends only on the sum of the state/context dimension and
control dimension, and does not distinguish between problems
with high state/context dimension and low control dimension
and vice versa.  Theorem~\ref{thm:sharp_lbd} also shows that
any one-step control algorithm that solves
Problem~\ref{prob_st:oc_p} will suffer from the \emph{curse
of dimensionality} --- the average regret
diminishes more slowly if either the context dimension or
the control dimension or both increase (see
Figure~\ref{fig:bounds}).  

Theorems~\ref{thm:sharp_ubd} and~\ref{thm:sharp_lbd} show
that the regret analysis of Algorithm~\ref{algo:cov_method}
with $\alpha=0$ and $\beta>0$ is tight with respect to $T$
up to the constants for $n_{\mathcal{Z}} + n_{\mathcal{U}}
\geq 3$.  As expected, the worst-case average 
regret of Algorithm~\ref{algo:cov_method}
diminish with increasing $T$.  

\begin{figure}
    \centering
    \includegraphics[width=1\linewidth]{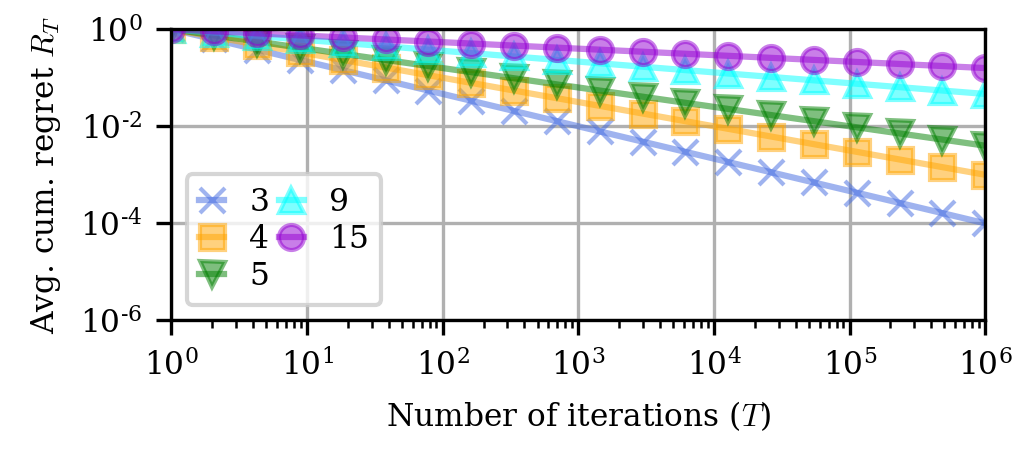} 
    \caption{Worst-case average 
        regret $R_T$ of
        Algorithm~\ref{algo:cov_method} ($\alpha=0$ and
        $\beta>0$) with varying number of
        maximum time steps $T$ for various choices of
        $n_{\mathcal{Z}} + n_\mathcal{U}$
        (Theorems~\ref{thm:sharp_ubd}
        and~\ref{thm:sharp_lbd}). For the sake of
        illustration, we have the normalized constants.}
    \label{fig:bounds}
\end{figure}

\section{Numerical experiments}

We show the efficacy of our approach using two numerical
experiments. In the first experiment, we compare the
performance of the proposed approach with existing
approaches to solve \eqref{prob:OC_P}. In the second
experiment, we demonstrate the utility of our approach in
landing an aircraft using data from a single finite-horizon
trajectory.

We used Python to perform all computations on an Intel
i7-4600U CPU with 4 cores, 2.1 GHz clock rate and 7.5 GB RAM.
We used \texttt{CVXPY}~\cite{cvxpy},
ECOS~\cite{domahidi2013ecos}, and
\texttt{GUROBI}~\cite{gurobi} for convex optimization,
\texttt{GPyOpt}~\cite{gpyopt2016} for Bayesian optimization,
and \texttt{optimize} from \texttt{scipy} for non-convex
optimization problems.

\begin{figure}[p]
    \newcommand{\trimValues}{0 20 0 20} \centering
    \includegraphics[width=0.83\linewidth,Trim=\trimValues,clip]{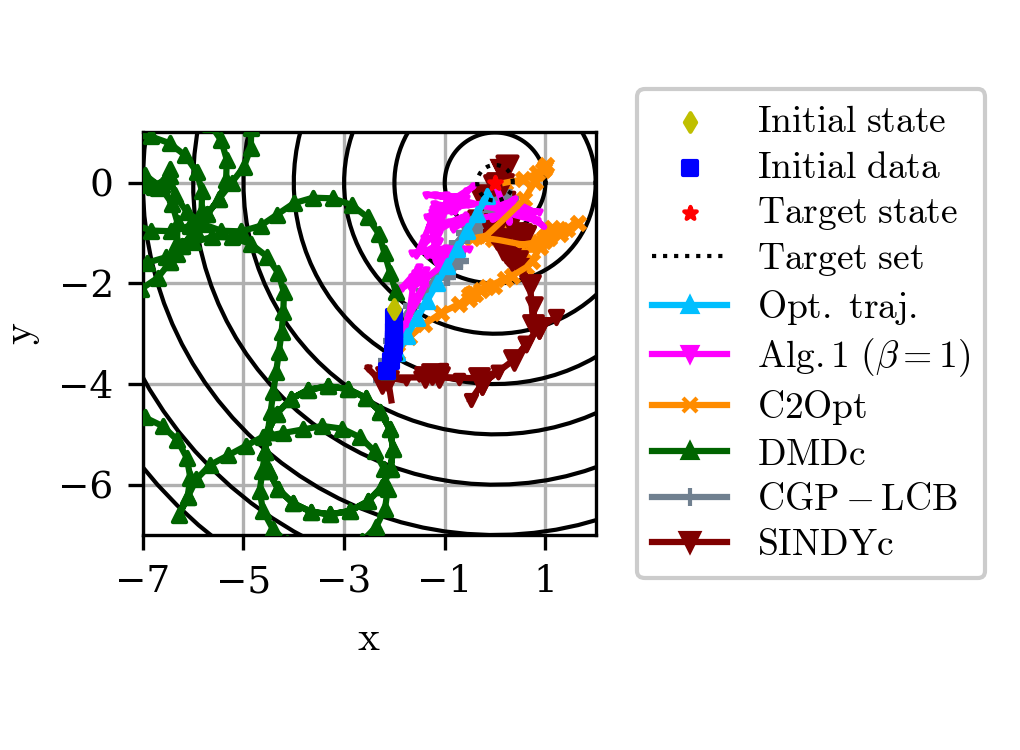}
    \includegraphics[width=0.85\linewidth]{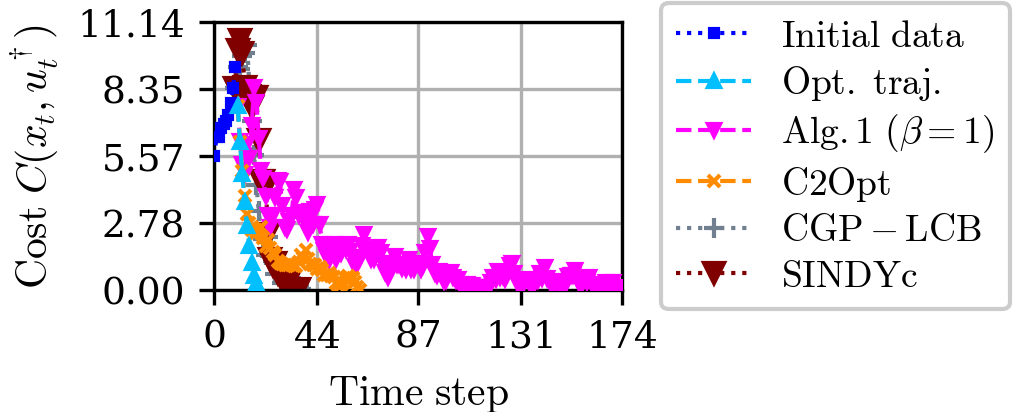}
    \includegraphics[width=0.85\linewidth]{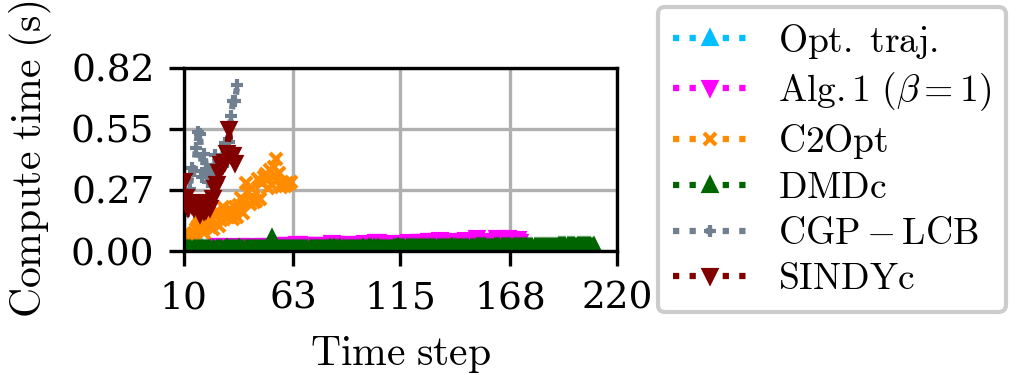}
    \caption{On-the-fly control for unicycle dynamics with
        $N=10$; \DMDc{}, Algorithm~\ref{algo:cov_method}
        ($\beta=1$), and \ACCalgo{} provide
        the faster controller synthesis, while
    \texttt{CGP-LCB}~\cite{krause2011contextual} and
\texttt{SINDYc}~\cite{kaiser2018sparse} reach the target in
fewer time steps. We omit \DMDc{} in plot of cost values for
sake of visualization of the one-step costs of other
methods.}
\label{fig:unicycle_traj_case1}
\end{figure}
\begin{figure}[p]
    \centering
    \newcommand{\trimValues}{0 15 0 25}
    \includegraphics[width=0.83\linewidth,Trim=\trimValues, clip]{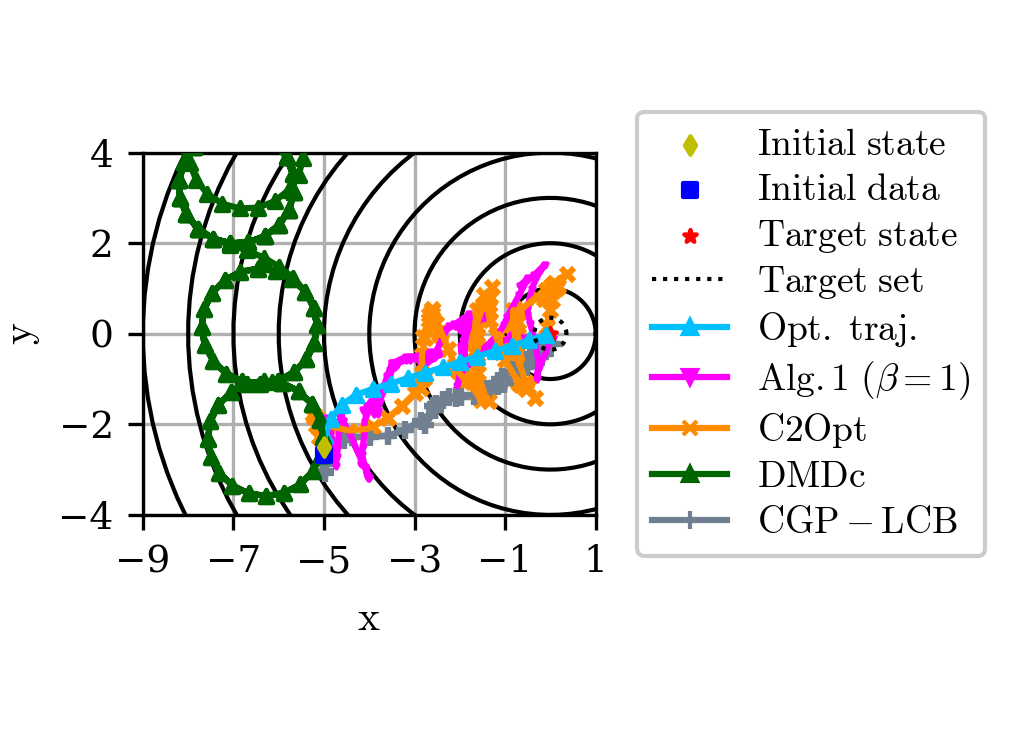}
    \includegraphics[width=0.85\linewidth]{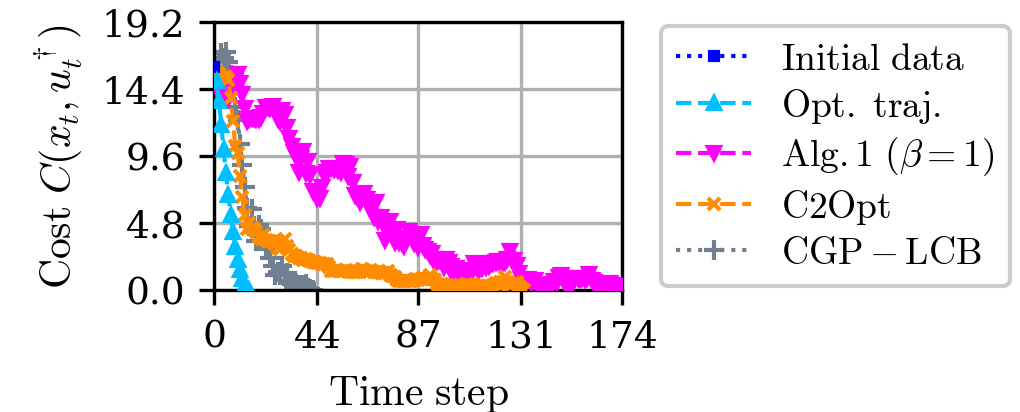}
    \includegraphics[width=0.85\linewidth]{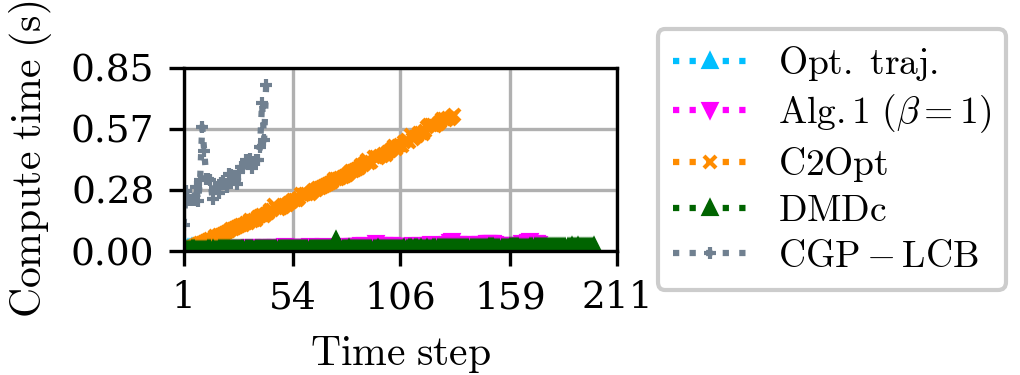}
    \caption{On-the-fly control for unicycle dynamics with
        $N=1$; Algorithm~\ref{algo:cov_method}
        ($\beta=1$), and \ACCalgo{} provide
        the faster controller synthesis, while
    \texttt{SINDYc}~\cite{kaiser2018sparse} fails at the
first time step. We omit \DMDc{} in plot of cost values for
sake of visualization of the one-step costs of other methods.}
\label{fig:unicycle_traj_case2}
\end{figure}

\begin{figure}
    \centering
    \newcommand{\trimValues}{0 10 0 0}
    \includegraphics[width=0.7\linewidth,
    Trim=\trimValues, clip]{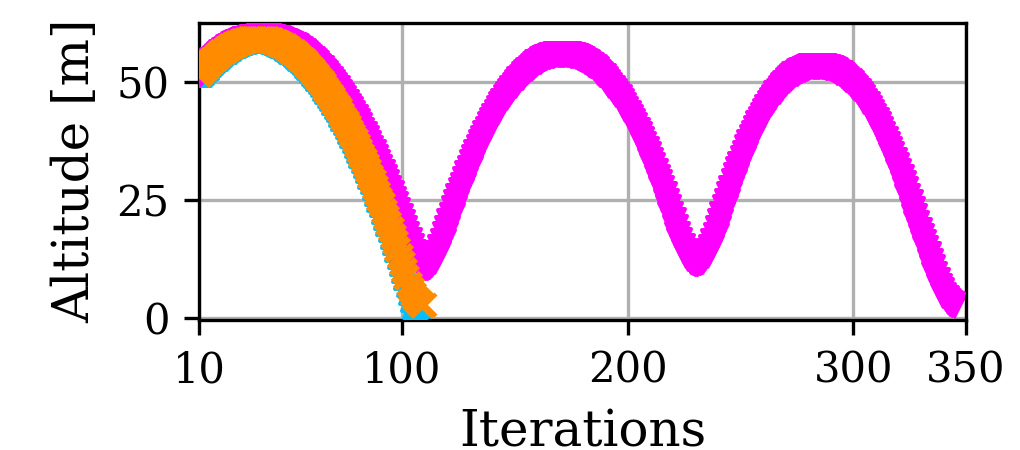}
    \includegraphics[width=0.7\linewidth]{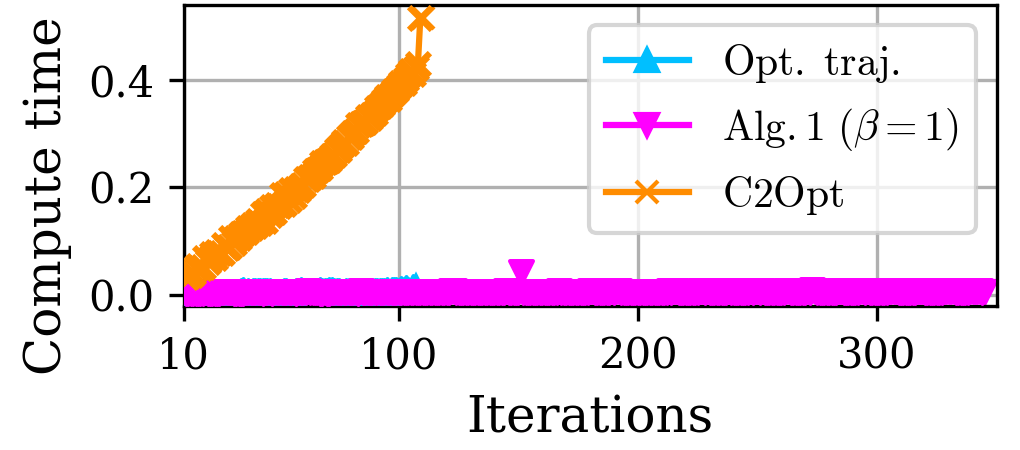}
    \caption{Damaged aircraft landing example: Altitude and
        compute time (s) over time.
 \ACCalgo{} (Algorithm~\ref{algo:cov_method} with 
$\alpha=\beta=0.5$) recovers the optimal trajectory, but has
a slower controller synthesis than Algorithm~\ref{algo:cov_method}
with $\beta=1$.}
\label{fig:aircraft_1}
\end{figure}
\begin{figure*}
    \centering
    \newcommand{\trimValuesLeft}{50 0 103 59}
    \newcommand{\trimValuesOthers}{70 0 83 59}
    \includegraphics[width=0.24\linewidth, Trim=\trimValuesLeft, clip]{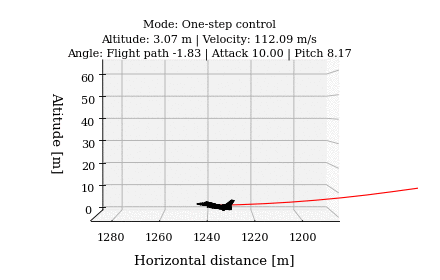} 
    \includegraphics[width=0.24\linewidth, Trim=\trimValuesOthers, clip]{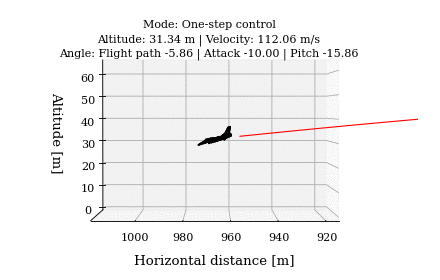} 
    \includegraphics[width=0.24\linewidth, Trim=\trimValuesOthers, clip]{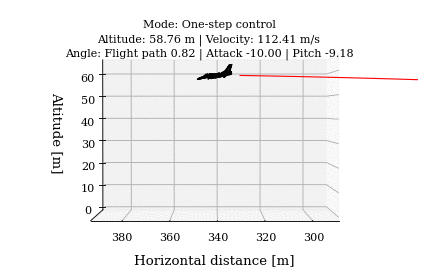} 
    \includegraphics[width=0.24\linewidth, Trim=\trimValuesOthers, clip]{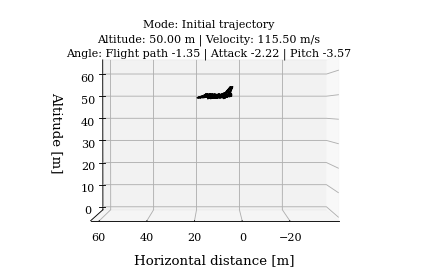}
    \caption{Landing of a damaged aircraft using \ACCalgo{}
        (Algorithm~\ref{algo:cov_method} with
    $\alpha=\beta=0.5$). The images are arranged with $t$
increasing from right to left.} \label{fig:aircraft_2}
\end{figure*}

\subsection{Unicycle dynamics: Comparison study}

Consider the problem of driving a unicycle system to origin
under control constraints. The true continuous-time dynamics
are given by,
\begin{align}
    \dot{p}_\mathrm{x} = v
    \cos(\theta),\quad\dot{p}_\mathrm{y}= v
    \sin(\theta),\quad\dot{\theta} =
    \omega,\label{eq:uni_cts_time}
\end{align}
with state $x=[p_\mathrm{x},p_\mathrm{y}, \theta]$ (position
in $\mathrm{x}$-$\mathrm{y}$ plane and heading) and
control input $
\overline{u}=[v, w]\in \mathcal{U}$. Using a zero-order hold
with a sampling time of $T_s=0.1$, we obtain a discrete-time
nonlinear dynamics $x_{t+1} = f(\overline{x}_t,
\overline{u}_t)$ for the purposes of simulation and data
generation. Since the state $x_t$ is non-Euclidean due
to the heading term $\theta$, we defined the context
$z={[p_\mathrm{x},p_\mathrm{y},\sin(\theta),\cos(\theta)]}^\top\in
\mathbb{R}^2\times [-1,1]^2$. We define the control input set as $
\mathcal{U}=[-4,4]\times[-\pi,\pi]$. We encode the goal of
driving the unicycle to the origin by seeking to regulate
the squared-distance of the unicycle position from the origin to zero. We define $C(z_t,u_t)={(p_\mathrm{x})}_{t+1}^2 +
{(p_\mathrm{y})}_{t+1}^2$ (see
\eqref{eq:one_step_ex}), and
choose $L_C=10$. We compare the proposed approach with 
$\alpha=\beta=0.5$ (refered to as \ACCalgo{}
in~\cite{vinod2020convexified}) and $\beta=1$ with
\texttt{CGP-LCB}~\cite{krause2011contextual},
\texttt{SINDYc}~\cite{kaiser2018sparse},
and \DMDc{}~\cite{korda2018linear}.
We relaxed the target state (origin) to a circle of radius $0.25$.
We applied \DMDc{} with $z_t=x_t$.  For more details on the
problem setup and choice of parameters for \texttt{SINDYc}
and \texttt{CGP-LCB}, see~\cite{vinod2020convexified}. 

We solve Problem~\ref{prob_st:oc_p} for two cases:
\begin{enumerate}
    \item initial state
        $x_0={[-2,-2.5,\pi/2]}^\top$ and $N=10$,
    and
    \item initial state $x_0 =
        {[-5,-2.5,\pi/2]}^\top$ and $N=1$.
\end{enumerate}
We generate the initial data using randomly generated
(uniform) control inputs for $N$ time steps respectively,
sampled from the set $[-4,0]\times [-\pi,\pi]\subset
\mathcal{U}$. The initial random trajectory drives the
unicycle away from the target by design.

Figures~\ref{fig:unicycle_traj_case1}
and~\ref{fig:unicycle_traj_case2} show the on-the-fly
control trajectories, the associated one-step costs, and the
compute times executed by an optimal one-step
controller that has knowledge of the underlying
dynamics (\texttt{Opt. traj.}), \texttt{CGP-LCB},
\DMDc{}, \texttt{SINDYc}
(when possible), and Algorithm~\ref{algo:cov_method}
(including \ACCalgo{}). We observe that the knowledge of
underlying dynamics are sufficient for one-step optimal
control to complete the task. Additionally, 
Algorithm~\ref{algo:cov_method} is superior to
\texttt{CGP-LCB} and \texttt{SINDYc} in computation time
required at each time step, while driving the unicycle
successfully to the target. On the other hand, \DMDc{} has
the least amount of computational costs, but it
fails to reach the target in both cases (potentially
due to the nonlinearity of the system, a one-step control
horizon setup, and severe limitation in data). Additionally, in Case 2 ($N=1$),  \texttt{SINDYc} suffers from numerical issues. In
both of the cases, \texttt{CGP-LCB} arrives at the goal in
the fewest time steps, but it has the highest computational
costs. Between two parameter choices of
Algorithm~\ref{algo:cov_method}, \texttt{C2Opt}
($\alpha=\beta$) arrives at the target faster than the
($\beta=1$) setting. This may be due to the fact that
Algorithm~\ref{algo:cov_method} guarantees sublinear
cumulative regret in the latter setting, and
therefore must spend more time in exploration phase.

\subsection{Landing a damaged aircraft}

Next, we consider the problem of landing a damaged aircraft.
For the purposes of simulation and collecting data, we use
the continuous-time, point mass longitudinal model of an
aircraft subject to the gravity force $mg$ with mass $m$,
acceleration due to gravity $g$, thrust $T$, lift $L$, and
drag $D$. The resulting equations of motion are,
\begin{subequations}
\begin{align}
    \dot{V}&=\frac{1}{m}(T\cos(\alpha) - D(\alpha, V)
    - mg\sin(\gamma))\\
    \dot{\gamma}&=\frac{1}{mV}(T\sin(\alpha) + L(\alpha, V)
    - mg\cos(\gamma))\\
    \dot{h}&=V\sin(\gamma)
\end{align}\label{eq:plane_dyn}%
\end{subequations}%
with states as the heading velocity $V$, the flight path
angle $\gamma$, and the altitude $h$, and control inputs as the
angle of attack $\alpha$ and thrust $T$. We use the
parameters for the DC9-30~\cite{bayen2007aircraft}, and
obtain the lift and drag forces are,
\begin{subequations}
\begin{align}
    L(\alpha, V) &= 68.6 (1.25 + 4.2 \alpha) V^2\\
    D(\alpha, V) &= (2.7 + 3.08(1.25 + 4.2 \alpha)^2) V^2
\end{align}%
\end{subequations}%
At an altitude of $h_0=50$ m, heading velocity $V_0 = 2 *
V_\mathrm{stall}=115.5$ m/s, we assume a damage on the aircraft
produces additive nonlinear terms $T\sin(\alpha)/(2m)$ and
$T\cos(\alpha)/(2mV)$ affecting the dynamics of velocity and
flight path angle. We discretize \eqref{eq:plane_dyn} in
time with a sampling time of $0.1$ s. We encode the goal of
landing the aircraft successfully by seeking to regulate the
altitude and flight path angle to zero. We define
$C(z_t,u_t)={h}_{t+1}^2 + {\alpha}_{t+1}^2$ (see
\eqref{eq:one_step_ex}), and select
$L_C=10$. We generate the initial data using randomly
generated control inputs for $N=10$ time steps. The
control input
space $ \mathcal{U}=[0.1, 0.3]\times (-10^\circ, 10^\circ)$,
with the thrust $T$ specified as a fraction of the maximum
thrust $T_{\max}=1.6\times 10^5$ N.

Figure~\ref{fig:aircraft_1} shows the regulation of the
altitude over time as well as the computation
time. Algorithm~\ref{algo:cov_method} completes the task
successfully for both choices of the parameter --- $\beta=1$
and $\alpha=\beta=0.5$ (\ACCalgo{}). We observe that
\ACCalgo{} recovers the optimal trajectory and takes
significantly fewer time steps than
Algorithm~\ref{algo:cov_method} with $\beta=1$. 
The ``oscillatory" closed-loop behavior and the resulting delay in reaching the target may be attributed to the fact that Algorithm~\ref{algo:cov_method} in $\beta=1$ setting spends more time in exploration phase in order to guarantee sublinear cumulative regret.
Figure~\ref{fig:aircraft_2} shows the different stages of
landing when using \ACCalgo{}.

\begin{rem}
    In both of the experiments considered, we found that 
    \texttt{GUROBI} was able to solve the non-convex
    quadratic program associated with
    Algorithm~\ref{algo:cov_method} faster than the convex
    programming-based approach
    \texttt{C2Opt}~\cite{vinod2020convexified}. We attribute
    the observed computational advantage to the heuristics
    used by \texttt{GUROBI}, and do not expect the advantage
    to hold for all one-step control instances of
    \eqref{prob:OC_P}.
\end{rem}

\section{Conclusion}

We propose a data-driven algorithm to solve contextual
optimization problems that arise in on-the-fly control of
unknown dynamical systems. We provide tractable formulations
for its implementation based on its convexity-like
properties. We establish that the proposed
approach enjoys sublinear cumulative regret with
respect to the number of control time steps, and prove that
the proposed regret analysis is tight. We also empirically
show that the proposed approach is typically much faster
than existing state-of-the-art, which enables real-time
application.

\appendix

\subsection{Proof of Lemma~\ref{lem:approx}}
\label{app:lem_approx}

    For any smooth function $f$, the following inequalities
    hold for any $s \in \mathcal{S}$,
    \begin{subequations}
        \begin{align}
            f(s) &\leq \ell(s;q_i,f) +
            \frac{L_f}{2}\|s - q_i\|^2,\qquad \forall i\in\Nint{t}
            \label{eq:approx_major_interim},\\
            f(s) &\geq \ell(s;q_i,f)
            - \frac{L_f}{2}\|s - q_i\|^2,\qquad \forall i\in\Nint{t}
            \label{eq:approx_minor_interim}.
        \end{align}\label{eq:approx_interim}%
    \end{subequations}%
    We obtain the data-driven majorant $f^+$
    \eqref{eq:C_plus_defn} and minorant $f^-$
    \eqref{eq:C_minus_defn}
    via finite minimum of \eqref{eq:approx_major_interim}
    and finite maximum of \eqref{eq:approx_minor_interim}
    over $i\in\Nint{t}$ respectively. By construction, these
    piecewise-quadratic functions coincide with $f(q_i)$ at
    $s=q_i$ for every $i\in\Nint{t}$. Also, \eqref{eq:approx_interim} implies
    \begin{subequations}
    \begin{align}
        f(s) - \left({\ell(s;q_i, f) -
    \frac{L_f}{2}\|s - q_i\|^2}\right)&\leq L_f\|s -
        q_i\|^2\label{eq:approx_minor_interim2},\\
        \left({\ell(s;q_i, f) +
        \frac{L_f}{2}\|s - q_i\|^2}\right) - f(s)&\leq L_f\|s -
        q_i\|^2\label{eq:approx_major_interim2}.
    \end{align}\label{eq:approx_interim2}%
    \end{subequations}
    We upper bound
    $f - f_t^-$ and $f^+_t - f$ by computing the finite minimum
    of the RHS in \eqref{eq:approx_minor_interim2} and
    \eqref{eq:approx_major_interim2} over $i\in\Nint{t}$.
    
\subsection{Proof of the volume counting lemma
(Lemma~\ref{lem:pidgeonhole})}
\label{app:lem_pidgeonhole}

    The set $\mathcal{S} \subset \mathbb{R}^d$ is covered by a
    hypercube of side $\diam(\mathcal{S})$.
    For any $\epsilon>0$, 
    the number of hypercubes of side
    $\sqrt{\frac{\epsilon}{d}}$ that covers the hypercube of
    side $\diam(\mathcal{S})$
    is given by $\left\lceil {\frac{\diam(
            \mathcal{S})^d}{\left(\epsilon/d\right)^\frac{d}{2}}
    }\right\rceil$. Note that $T$ is at least one more than
    this minimum number. By the \emph{pidgeonhole principle}, at
    least one of the hypercubes with side
    $\sqrt{\frac{\epsilon}{d}}$ must have at least two
    points. However, the maximum separation allowed between two
    points within such a hypercube is $\sqrt{\epsilon}$. Thus,
    for some $i,j\in\Nint{T}$ with $i\neq j$, we have
    $\|q_i - q_j\|\leq \sqrt{\epsilon}$. We complete the
    proof with $t\triangleq\max(i,j)-1\in\Nint{T-1}$. \qed

\subsection{Proof of Lemma~\ref{lem:ibp2}}
\label{app:lem_ibp2}

We first introduce an auxiliary result (Lemma~\ref{lem:ibp})
that will help in proving Lemma~\ref{lem:ibp2}.

\begin{lem}\label{lem:ibp}
Let $\mu > 1$ and $T>0$. Suppose $h : [0,\infty) \rightarrow [0,\infty)$ is an integrable function satisfying $\int_0^\infty h(t) dt = T$ and $h(t) \leq e^{\mu t}$ for all $t$. Then
\[
\int_0^\infty h(t) e^{-t} dt \leq \mu (\mu-1)^{-1} T^{1-1/\mu}.
\]
\end{lem}
\begin{proof}
If $T < 1$, we have $T\leq T^{1-\frac{1}{\mu}}$ and
\begin{align}
\int_0^\infty h(t) e^{-t} dt \leq \int_0^\infty h(t) dt = T
\leq \frac{\mu}{\mu-1} T^{1-1/\mu}.
\label{eq:lem_ibp2_tleq1}
\end{align}

Now suppose $T \geq 1$. Set $H(t) = \int_0^t h(s) ds$. Observe that $H(t) \leq T$ for all $t$, and $H(0) = 0$. Using integration by parts,
\[
\begin{aligned}
\int_0^\infty h(t) e^{-t} dt &= H(0) - \lim_{t \rightarrow \infty} H(t) e^{-t} + \int_0^\infty H(t) e^{-t} dt \\
&= \int_0^\infty H(t) e^{-t} dt.
\end{aligned}
\]
For $t>0$, we have $H(t) \leq \int_0^{t} e^{\mu s} ds =
\mu^{-1} (e^{\mu t} - 1) \leq \mu^{-1} e^{\mu
t}$. Set $t^* := \ln(\mu T)/\mu > 0$. Then
\begin{align}
\int_0^\infty & H(t) e^{-t} dt \nonumber \\
              &\leq \mu^{-1}  \int_0^{t^*} e^{\mu t} e^{-t}
              dt + \int_{t^*}^\infty T e^{-t} dt
              \label{eq:main_term}\\
              & = \mu^{-1} (\mu-1)^{-1} (e^{(\mu-1) t^*} -
              1) + T e^{-t^*} \nonumber \\
              &=  \mu^{-1} (\mu-1)^{-1} ((\mu
              T)^{(\mu-1)/\mu} - 1) + T (\mu T)^{-1/\mu} \nonumber \\
              &\leq \mu^{-1/\mu} (1+(\mu-1)^{-1})
              T^{1-1/\mu} \leq \mu (\mu-1)^{-1} T^{1-1/\mu},\nonumber
\end{align}
as desired.
\end{proof}

We are now ready to prove Lemma~\ref{lem:ibp2}.

By normalization, we may assume $A=1$. Indeed, if we make the substitutions $n_k/A \rightarrow n_k$ and $T/A \rightarrow T$, both the hypotheses and conclusion of Lemma~\ref{lem:ibp2} will remain valid.

Define a step function $h : [0,\infty) \rightarrow [0,\infty)$, where $h(t) = 0$ for $t \in [0,1)$, and $h(t) = n_k$ for $t \in [k,k+1)$, $k \in \mathbb{N}$. For $t \in [k,k+1)$, $k \in \mathbb{N}$, we have
\[
h(t) = n_k \leq e^{\mu k} \leq e^{\mu t}.
\]
Furthermore, we have $\int_0^\infty h(t) dt = \sum_{k=1}^\infty n_k = T$. We may thus apply Lemma \ref{lem:ibp} to the function $h$. We find that
\[
\int_0^\infty h(t) e^{-t} dt \leq \mu (\mu-1)^{-1} T^{1-1/\mu}.
\]
But then, since $e^{-t}$ is decreasing,
\[
\int_0^\infty h(t) e^{-t} dt = \sum_{k=1}^\infty n_k \int_{k}^{k+1} e^{-t} dt \geq \sum_{k=1}^\infty n_k e^{-k}.
\]
Hence, combining the previous inequalities,
\[
\sum_{k=1}^\infty n_k e^{-k} \leq \mu (\mu-1)^{-1} T^{1-1/\mu}.
\]
This completes the proof.
\qed

\subsection{Proof of Theorem~\ref{thm:sharp_lbd}}
\label{app:thm_sharp_lbd}

We first introduce three auxiliary results
(Lemmas~\ref{lem:whitneyB},~\ref{lem:whitneyC},
and~\ref{lem:disjointball}) that will help
in proving Theorem~\ref{thm:sharp_lbd}.  

Since the context map is an identity map, we have $z_t=x_t$
at every time step $t$, $ \mathcal{Z}= \mathcal{X}$, and
$n_{\mathcal{Z}}=n_{ \mathcal{X}}$. We continue the
derivation with $ \mathcal{Z}$ to remain consistent in
notation.

For brevity, we define $d=n_\mathcal{Z} + n_\mathcal{U}$.
For any subset $\mathcal{S} \subset \mathbb{R}^d$ and $s \in
\mathbb{R}^d$, we define the distance function as
$\dist(s,\mathcal{S}) = \inf_{w \in \mathcal{S}} \| w - s
\|$.

\begin{lem}[Thm. 2 of Ch. VI, Sec.  2.1
    in \cite{stein1970singular}]\label{lem:whitneyB}
Given a closed set $\mathcal{S} \subset \mathbb{R}^d$, there
exists a non-negative function $\rho : \mathbb{R}^d \setminus \mathcal{S}
\rightarrow [0,\infty)$ with continuous derivatives up to
all orders ($\rho$ is a $C^\infty$ function), and constants
$m_1,m_2 > 0$ determined by $d$ such that $$m_1
\dist(s,\mathcal{S}) \leq \rho(s) \leq m_2
\dist(s,\mathcal{S}).$$ 
For all multiindices $\alpha$ and constants $C_\alpha$ that depend only on $d$ and $\alpha$, the higher
derivatives satisfy 
$\left|{\frac{\partial^\alpha}{\partial
s^\alpha}\rho(s)}\right| \leq C_\alpha
\dist(s,\mathcal{S})^{1 - |\alpha|}$ for all $s \in
\mathbb{R}^d \setminus \mathcal{S}$.
\end{lem}

The function $\rho$ in the previous lemma is called a
``Whitney regularized distance function'' for $\mathcal{S}$.
We use $\|\cdot\|_{op}$ to denote the operator norm.

\begin{lem}\label{lem:whitneyC}
    Given $L_C > 0$, $\tau\in \mathbb{N}$, and a closed set
    of finite points $\mathcal{S}= {\{(z_t,u_t)\}}_{t \in [\tau]}
    \subset\mathbb{R}^d$, there exists a smooth function $C
    : \mathbb{R}^d \rightarrow (-\infty,0]$ such that $K_C
    \leq L_C$, $C(z_t,u_t)=0$ and $\nabla C(z_t,u_t)=0$ for
    every $t\in\Nint{\tau}$, and 
    $C(s )\leq -\kappa L_C \dist(s, \mathcal{S})^2 $ and
    $\|\nabla C(s)\|\leq L_C \diam(\mathcal{Z} \times
    \mathcal{U})$ for all $s \in \mathbb{R}^d$, where $\kappa >
    0$ is a constant that depends only on $d$.
\end{lem}
\begin{proof}
Let $\rho: \mathbb{R}^d \setminus \mathcal{S} \rightarrow
[0,\infty)$ be the Whitney regularized distance function for
$\mathcal{S}$. Then, by Lemma~\ref{lem:whitneyB}, we have some
positive constants $m_1, m_2, M_1$, and $M_2$ such that the
following observations hold for all $s \in \mathbb{R}^d \setminus \mathcal{S}$:
a) $m_1
\dist(s,\mathcal{S}) \leq \rho(s) \leq m_2
\dist(s,\mathcal{S})$, b) $\|\nabla \rho(s)\| \leq M_1$ and
c) $\|
\nabla^2 \rho(s)\|_{op} \leq M_2 \dist(s,\mathcal{S})^{-1}$.

Consider the function $\eta: \mathbb{R}^d \rightarrow
\mathbb{R}$, given by $\eta(s_t) = 0$ for all $t$, and
$\eta(s)  = \rho(s)^2$ for $s \in \mathbb{R}^d \setminus
\mathcal{S}$. We know $\eta$ is a continuously
differentiable function and $\|\nabla \eta(s_t)\|= 0$ for all
$t\in\Nint{\tau}$, since for some positive constants
$M_3,M_4$, $\rho(s) \leq M_3\|s-s_t\|$  implies $\eta(s)
\leq M_4\|s-s_t\|^2$ as $s\to s_t$ for any $t \in
\Nint{\tau}$ by a). By the product rule, $
\nabla^2 \eta(s) = 2 \nabla \rho(s) \otimes \nabla \rho(s) +
2 \rho(s) \nabla^2 \rho(s) $ for all $s \in \mathbb{R}^d
\setminus \mathcal{S}$. Thus, $\nabla^2 \eta(s)$ is
twice-differentiable with 
\[
\begin{aligned}
\|\nabla^2 \eta(s)\|_{op} &\leq 2 \| \nabla \rho(s) \|^2 + 2 \rho(s) \| \nabla^2 \rho(s) \|_{op} \\
& \leq 2 M_1^2 + 2 m_2 \dist(s,\mathcal{S}) M_2
\dist(s,\mathcal{S})^{-1} \\
&= 2M_1^2 + 2m_2M_2\triangleq M_5.
\end{aligned}
\]
By~\cite[Lem. 1.2.2]{nesterov2018lectures}, $\nabla \eta$ is
$M_5$-Lipschitz. Note that $M_5$ depends only on $d$.

Finally, define $C(s) = - (L_C/M_5) \eta(s)$. Then $\nabla
C$ is $L_C$-Lipschitz with $C(s_t) = 0$ and $\nabla C(s_t) =
0$ for all $0 \leq t \leq \tau$, by construction.  By the
Mean Value Theorem, we also have $\|\nabla C(s)\| \leq L_C
\diam( \mathcal{Z}\times \mathcal{U})$, as desired. 
From property a) of $\rho$ and definition of $\eta$, $
\eta(s) \geq m_1^2
\dist(s,\mathcal{S})^2$ for all $s \in \mathbb{R}^d$. 
Consequently,  $C(s) \leq - \kappa L_C
\dist(s,\mathcal{S})^2$ for all $s \in \mathbb{R}^d$ with
$\kappa = m_1^2/M_5$. Here, $\kappa$ depends only on $d$.
\end{proof}

\newcommand{\VolBallnu}{V(n_\mathcal{U})}
Let $\VolBallnu$ denote the volume of the unit ball in
$n_{\mathcal{U}}$-dimensional Euclidean space, $
\mathrm{Vol}( \mathcal{S})$ denote the volume of a set $
\mathcal{S}$, and define $\mathrm{Ball}(c, r)=\{s\in
\mathbb{R}^d:\|s -c \| \leq r\}\subset \mathbb{R}^d$.
\begin{lem}\label{lem:disjointball} Given $\xi$ and $T$. Let $ \mathcal{Z}$ and $
    \mathcal{U}$ have strictly positive diameters, and
    $\mathcal{S}= {\{(z_t,u_t)\}}_{t \in [T]} \subset
    \mathbb{R}^d$ be a sequence of context-control pairs such
    that $\{z_t\}_{t \in [T]}$ is $\xi$-separated, i.e., $\|
    z_t - z_{t'} \| \geq \xi$ for $t,t'\in \Nint{T}$ and $t
    \neq t'$. For every ball
    $B_{\mathcal{Z}}\subset\mathcal{Z}$ of radius $\delta
    \leq {\left({\omega \VolBallnu
    \xi^{n_{\mathcal{Z}}}}\right)}^{1/d}$ with 
    $\omega\triangleq (2
    \sqrt{n_{\mathcal{Z}}})^{n_{\mathcal{Z}}} + 1$, there
    exists a ball $B_{\mathcal{U}}\subset\mathcal{U}$ of
    radius $\delta$ such that every point in $ \mathcal{S}$
    lies outside $B_{\mathcal{Z}} \times B_{\mathcal{U}}$.
    In other words, $ \mathcal{S}\cap (B_{\mathcal{Z}}\times
    B_{\mathcal{U}})=\emptyset$.
\end{lem}
\begin{proof}
    Given $B_{\mathcal{Z}}$, define the set $\mathcal{I}
    \subset [T]$ by $\mathcal{I} = \{ t \in [T] : z_t \in
    B_{\mathcal{Z}} \}$, the time steps of the contexts in $
    \mathcal{S}$ belonging to $B_{\mathcal{Z}}$ of radius
    $\delta$. 
    \newcommand{\OmegaCover}{\bigcup_{t \in \mathcal{I}}
    \mathrm{Ball}(u_t, \delta)}

    To show that $ \mathcal{S}\cap (B_{\mathcal{Z}}\times
    B_{\mathcal{U}})=\emptyset$, we claim that, for the
    given choice of $\delta$, there always exists
    $v \in \mathcal{U} \setminus \OmegaCover$, i.e., the set
    $\mathcal{U} \setminus \OmegaCover$ is non-empty. Since $v \notin
    \OmegaCover$, we have $\| u_t - v \| > \delta$ for all $t
    \in \mathcal{I}$, which implies that $u_{t} \notin
    B_{\mathcal{U}}\triangleq \mathrm{Ball}(v,\delta)$
    for all $t \in \mathcal{I}$. On the other hand, for $t
    \in [T] \setminus \mathcal{I}$, $z_t \notin
    B_{\mathcal{Z}}$ by the definition of $ \mathcal{I}$.
    Consequently, for any $t \in [T]$, we have $(z_t,u_t)
    \notin B_{\mathcal{Z}} \times B_{\mathcal{U}}$.

    We complete the proof by proving our claim that
    $\mathcal{U}\setminus\OmegaCover$ is non-empty. We show that
    $\mathrm{Vol}(\OmegaCover) < \mathrm{Vol}(\mathcal{U})$ for
    the given choice of $\delta$, which in turn implies the
    claim. Note that
    $| \mathcal{I}|\leq \omega
    (\delta/\xi)^{n_{\mathcal{Z}}}$ by using the
    contrapositive of the volume counting lemma
    (Lemma~\ref{lem:pidgeonhole}) within the ball
    $B_{\mathcal{Z}}$ and utilizing the assumption that the
    context sequence $z_t$ is $\xi$-separated. Then, by the
    choice of $\delta$, 
    \begin{align}
        \mathrm{Vol}\left({\OmegaCover}\right) &\leq \VolBallnu |\mathcal{I}|
    \delta^{n_{\mathcal{U}}} \nonumber \\
    &\leq \VolBallnu\omega 
    (\delta/\xi)^{n_{\mathcal{Z}}} \delta^{n_{\mathcal{U}}}
    \nonumber \\
                           &= \omega \VolBallnu \delta^{d}
                           \xi^{-n_{\mathcal{Z}}}\leq
                           \mathrm{Vol}( \mathcal{U}). \nonumber
    \end{align}
    This completes the proof.
\end{proof}

We are now ready to prove Theorem~\ref{thm:sharp_lbd}.

    We note that there always exists dynamics $F$ for
    some state space $ \mathcal{X}$ and control input set $
    \mathcal{U}$ that satisfies the first requirement of
    $(\xi,T)$-resistance.  Specifically, there exists $F$
    that generates a sequence of states
    $\{x_t\}_{t\in\Nint{T}}$ starting from $x_0$, such that
    $\|x_{t+1} - x_i\| > \xi$ for every $i\in \Nint{t}$,
    irrespective of the control inputs selected by
    $\mathscr{A}$. For example, when $\mathcal{X}$ is an
    axis-aligned hypercube, the given choice of $\xi$
    fits a grid of $T$ points in $\mathcal{X}$ with grid spacing of
    $\sqrt{\frac{\xi}{n_{\mathcal{X}}}}$ along each dimension. In this case, $F$ can be defined
    as a nonlinear dynamics that visits each of these
    grid points only once to ensure that the resulting
    trajectory has $\xi$-separation. Since
    $\Phi$ is an identity map,  
    $\|z_{t+1} - z_i\| > \xi$ for every $i\in \Nint{t}$.
    
    Let ${\{u_t\}}_{t=0}^{T-1}$ be the sequence of
    corresponding control actions chosen by the control
    algorithm $\mathscr{A}$. Consider the sequence of
    context-control pairs $\mathcal{S} = \{ (z_t,u_t)
    \}_{t=0}^{T-1} \subset \mathbb{R}^{d}$, $d= n_{\mathcal{Z}}
    + n_{\mathcal{U}}$. Given $L_C > 0$, there exists a
    $L_C$-smooth $C$ that satisfies the requirements of a
    $(\xi,T)$-resisting problem instance of
    \eqref{prob:OC_P} for the control algorithm $
    \mathscr{A}$, and
    Assumptions~\ref{assum:nonoise}--~\ref{assum:norm_bound}
    with $\GCmax=L_C\diam( \mathcal{Z}\times \mathcal{U})$
    by Lemma~\ref{lem:whitneyC}. Thus, the
    $(\xi,T)$-resisting problem instance of
    \eqref{prob:OC_P} is well-defined.
    
    Fix $\delta ={\left({\omega \VolBallnu
    \xi^{n_{\mathcal{Z}}}}\right)}^{1/d}$, with $\omega$ as
    in Lemma~\ref{lem:disjointball}, and apply
    Lemma~\ref{lem:disjointball} with the ball
    $B_{\mathcal{Z}} = B(z_t, \delta)$ for some $t \in [T]$.
    So there exists $v_t \in \mathcal{U}$ such that
    $\mathcal{S} \cap (B(z_t, \delta) \times B(v_t,\delta))
    = \emptyset$. Consequently, $\dist((z_t,v_t),
    \mathcal{S}) \geq \delta$ for every $t\in\Nint{T}$. Using
    Lemma~\ref{lem:whitneyC}, we have 
    \[
    C(z_t,v_t) \leq - \kappa_1 L_C \dist((z_t,v_t),
    \mathcal{S})^2 \leq - \kappa_1 L_C \delta^2.
    \]
    Meanwhile, $C(z_t,u_t) = 0$ by design. Hence, the
    regret incurred by the
    control algorithm $\mathscr{A}$ is bounded from below for
    each $t \in [T]$,
    \[
    \rho_t \geq C(z_t,u_t) - C(z_t,v_t) \geq \kappa_1 L_C \delta^2 = \kappa_1 L_C {\left({\omega \VolBallnu
    \xi^{n_{\mathcal{Z}}}}\right)}^{2/d}.
    \]
    By the choice of $\xi$, 
    $\rho_T \geq \nu T^{-2/d}$ and $R_T
    \geq \nu T^{-2/d}$ for some positive constant $\nu$
    that depends on $L_C, n_{ \mathcal{Z}}$, and $n_{
    \mathcal{U}}$.\qed

\subsection{Regret analysis for
Algorithm~\ref{algo:cov_method} with $n_{\mathcal{Z}} +
n_{\mathcal{U}} = 2$}
\label{app:diff_lem}

We focus on the regret analysis for
Algorithm~\ref{algo:cov_method} with $n_{\mathcal{Z}} +
n_{\mathcal{U}} = 2$ (the case omitted by
Theorem~\ref{thm:sharp_ubd}).

\begin{prop}[\textsc{Algorithm~\ref{algo:cov_method} has
sublinear cumulative regret, even when $n_{\mathcal{Z}} +
n_{\mathcal{U}} = 2$}]\label{prop:sharp_ubd_2}
    Suppose $n_{\mathcal{Z}} + n_{\mathcal{U}} = 2$ and
    $N=0$. Then, the cumulative regret of
    Algorithm~\ref{algo:cov_method} with $\alpha=0$ and
    $\beta>0$ is sublinear, and its average regret $R_T \leq
    \widehat{M}_1 \log(T)/T + \widehat{M}_2/T$ for 
    $T\geq T_0$. Here,
    $T_0$, $\widehat{M}_1$, and $\widehat{M}_2$ are
    positive constants that depend on $L_C$, $\GCmax$,
    $\alpha$, $\beta$, and the diameters and dimensions of
    the sets $\mathcal{Z}$ and $\mathcal{U}$
\end{prop}

The proof of Proposition~\ref{prop:sharp_ubd_2} is similar
to that of Theorem~\ref{thm:sharp_ubd}.  Recall the
definition of $n_k$, $t^\ast$, and $\mu$ from the
proof of Theorem~\ref{thm:sharp_ubd} and
Lemmas~\ref{lem:ibp2} and~\ref{lem:ibp}. In \eqref{eq:regret_ub_last_step}, we
had used Lemma~\ref{lem:ibp2} to upper bound a series 
$\sum_{k=1}^\infty n_k e^{-k}$. However, when proving
Proposition~\ref{prop:sharp_ubd_2}, we can not use the same
bound since $\mu=1$. 

Instead, we prove a version of Lemma~\ref{lem:ibp2} for $\mu=1$. We show that for $\{n_k\}$ satisfying $n_k \leq A e^k$ and $\sum_{k=1}^\infty n_k = T$, we have
\begin{align}
        \sum_{k=1}^\infty n_k e^{-k} 
        \leq 
        \begin{cases}
        T,  &  T < A,\\
        A\left({\ln\left(\frac{T}{A}\right) + 1}\right), & \text{otherwise}\\
    \end{cases}\label{eq:new_bound_n_k_series}
\end{align}
Similarly to the proof of Lemma~\ref{lem:ibp2}, we make the substitutions $n_k/A \rightarrow n_k$ and $T/A \rightarrow T$. Thus, we may suppose $A=1$ in the above, without loss of any generality. 

If $T < 1$, we have $\sum_{k=1}^\infty n_k e^{-k} \leq \sum_{k=1}^\infty n_k \leq T$. On the other hand,
if $T\geq 1$ then we define a step function $h(t) = n_k$ for $t \in [k,k+1)$, $k \in \mathbb{N}$ and $h(t) = 0$ for $t \in [0,1)$. As in the proof of Lemma~\ref{lem:ibp2}, it suffices to show $\int_0^\infty h(t) e^{-t} \leq \ln(T) + 1$. Set $t^\ast=\ln(T)$, $H(t) = \int_0^t h(s) ds$, and follow the proof of Lemma~\ref{lem:ibp} in the case $\mu=1$. As before, integration by parts shows $\int_0^\infty h(t) e^{-t} dt = \int_0^\infty H(t) e^{-t} dt$. Note that
\eqref{eq:main_term} simplifies as follows, since $\mu=1$, $t^\ast=\ln(T)$,
\begin{align}
        \int_0^\infty H(t) e^{-t} dt \leq \int_0^{t^*}
         dt + Te^{-t^\ast} =
        \ln(T) + 1. \nonumber
\end{align}
Finally, we complete the proof of
Proposition~\ref{prop:sharp_ubd_2} by applying
\eqref{eq:new_bound_n_k_series} to
\eqref{eq:regret_ub_last_step} for $T \geq T_0 = M_0/
(\GCmax \diam(\mathcal{U}))$ (see the proof of
Theorem~\ref{thm:sharp_ubd} for the definition of $M_0$).

\bibliographystyle{IEEEtran}
\bibliography{IEEEabrv,shortIEEE,ACC2020_refs} 

\begin{thebibliography}{10}
\providecommand{\url}[1]{#1}
\csname url@samestyle\endcsname
\providecommand{\newblock}{\relax}
\providecommand{\bibinfo}[2]{#2}
\providecommand{\BIBentrySTDinterwordspacing}{\spaceskip=0pt\relax}
\providecommand{\BIBentryALTinterwordstretchfactor}{4}
\providecommand{\BIBentryALTinterwordspacing}{\spaceskip=\fontdimen2\font plus
\BIBentryALTinterwordstretchfactor\fontdimen3\font minus
  \fontdimen4\font\relax}
\providecommand{\BIBforeignlanguage}[2]{{%
\expandafter\ifx\csname l@#1\endcsname\relax
\typeout{** WARNING: IEEEtran.bst: No hyphenation pattern has been}%
\typeout{** loaded for the language `#1'. Using the pattern for}%
\typeout{** the default language instead.}%
\else
\language=\csname l@#1\endcsname
\fi
#2}}
\providecommand{\BIBdecl}{\relax}
\BIBdecl

\bibitem{ornik2019myopic}
M.~Ornik, S.~Carr, A.~Israel, and U.~Topcu, ``Myopic control of systems with
  unknown dynamics,'' in \emph{Proc. American Ctrl. Conf.}, 2019, pp.
  1064--1071.

\bibitem{djeumou2020fly}
F.~Djeumou, A.~Vinod, E.~Goubault, S.~Putot, and U.~Topcu, ``On-the-fly control
  of unknown systems: From side information to performance guarantees through
  reachability,'' \emph{arXiv preprint arXiv:2009.12733}, 2020.

\bibitem{krause2011contextual}
A.~Krause and C.~Ong, ``Contextual gaussian process bandit optimization,'' in
  \emph{Adv. {N}eural {I}nfo. {P}roc. {S}yst.}, 2011, pp. 2447--2455.

\bibitem{chowdhury2017ICML}
S.~Chowdhury and A.~Gopalan, ``On kernelized multi-armed bandits,'' in
  \emph{Proc. Int'l Conf. Machine Learning}, 2017, pp. 844--853.

\bibitem{valko2013finite}
M.~Valko, N.~Korda, R.~Munos, I.~Flaounas, and N.~Cristianini, ``Finite-time
  analysis of kernelised contextual bandits,'' in \emph{Conf. Uncertainity in
  Artificial Intelligence}, 2013.

\bibitem{berkenkamp2016safe}
F.~Berkenkamp, A.~P. Schoellig, and A.~Krause, ``Safe controller optimization
  for quadrotors with {G}aussian processes,'' in \emph{Proc. IEEE Int'l Conf.
  Robotics and Autom.}, 2016, pp. 491--496.

\bibitem{bogunovic2020corruption}
I.~Bogunovic, A.~Krause, and J.~Scarlett, ``Corruption-tolerant gaussian
  process bandit optimization,'' in \emph{International Conference on
  Artificial Intelligence and Statistics}.\hskip 1em plus 0.5em minus
  0.4em\relax PMLR, 2020, pp. 1071--1081.

\bibitem{vinod2020convexified}
A.~Vinod, A.~Israel, and U.~Topcu, ``Convexified contextual optimization for
  on-the-fly control of smooth systems,'' in \emph{Proc. Amer. Ctrl. Conf.},
  2020, pp. 2004--2011.

\bibitem{ahmadi2020learning}
A.~Ahmadi and B.~El~Khadir, ``Learning dynamical systems with side
  information,'' in \emph{Learning for Dynamics and Control}.\hskip 1em plus
  0.5em minus 0.4em\relax PMLR, 2020, pp. 718--727.

\bibitem{ahmadi2020safe}
M.~Ahmadi, A.~Israel, and U.~Topcu, ``Safe controller synthesis for data-driven
  differential inclusions,'' \emph{IEEE Transactions on Automatic Control},
  vol.~65, no.~11, pp. 4934--4940, 2020.

\bibitem{korda2018linear}
M.~Korda and I.~Mezi{\'c}, ``Linear predictors for nonlinear dynamical systems:
  {K}oopman operator meets model predictive control,'' \emph{Automatica},
  vol.~93, pp. 149--160, 2018.

\bibitem{kaiser2018sparse}
E.~Kaiser, J.~Kutz, and S.~Brunton, ``Sparse identification of nonlinear
  dynamics for model predictive control in the low-data limit,'' \emph{Proc. of
  the Royal Soc. A}, vol. 474, no. 2219, 2018.

\bibitem{proctor2016dynamic}
J.~Proctor, S.~Brunton, and N.~Kutz, ``Dynamic mode decomposition with
  control,'' \emph{J. App. Dyn. Syst.}, vol.~15, no.~1, pp. 142--161, 2016.

\bibitem{chowdhary2014bayesian}
G.~Chowdhary, H.~A. Kingravi, J.~P. How, and P.~Vela, ``Bayesian nonparametric
  adaptive control using gaussian processes,'' \emph{IEEE Trans. Neural
  Networks \& Learning Syst.}, vol.~26, no.~3, pp. 537--550, 2014.

\bibitem{calliess2014conservative}
J.-P. Calliess, ``Conservative decision-making and inference in uncertain
  dynamical systems,'' Ph.D. dissertation, Oxford Univ., 2014.

\bibitem{an2005dc}
L.~An and P.~Tao, ``The {DC} (difference of convex functions) programming and
  {DCA} revisited with {DC} models of real world nonconvex optimization
  problems,'' \emph{Annals of Op. Research}, vol. 133, no.~1, pp. 23--46, 2005.

\bibitem{nesterov2018lectures}
Y.~Nesterov, \emph{Lectures on convex optimization}.\hskip 1em plus 0.5em minus
  0.4em\relax Springer, 2018.

\bibitem{kochenderfer2019algorithms}
M.~J. Kochenderfer and T.~A. Wheeler, \emph{Algorithms for optimization}.\hskip
  1em plus 0.5em minus 0.4em\relax MIT Press, 2019.

\bibitem{kocijan2004gaussian}
J.~Kocijan, R.~Murray-Smith, C.~Rasmussen, and A.~Girard, ``Gaussian process
  model based predictive control,'' in \emph{Proc. American Ctrl. Conf.}, 2004,
  pp. 2214--2219.

\bibitem{bertsekas2019reinforcement}
D.~Bertsekas, \emph{Reinforcement learning and optimal control}.\hskip 1em plus
  0.5em minus 0.4em\relax Athena Scientific, 2019.

\bibitem{vinod2020constrained}
A.~Vinod, A.~Israel, and U.~Topcu, ``Constrained, global optimization of
  functions with {L}ipschitz continuous gradients,'' \emph{SIAM J. Opt.},
  vol.~32, no.~2, pp. 1239--1264, 2022.

\bibitem{horst2000introduction}
R.~Horst, P.~Pardalos, and N.~Thoai, \emph{Introduction to global
  optimization}.\hskip 1em plus 0.5em minus 0.4em\relax Springer Science \&
  Business Media, 2000.

\bibitem{pardalos2010deterministic}
P.~M. Pardalos, Q.~P. Zheng, and A.~Arulselvan, ``Deterministic global
  optimization,'' \emph{Encycl. Oper. Res. Management Sci.}, 2010.

\bibitem{mokhtari2016online}
A.~Mokhtari, S.~Shahrampour, A.~Jadbabaie, and A.~Ribeiro, ``Online
  optimization in dynamic environments: Improved regret rates for strongly
  convex problems,'' in \emph{IEEE Conf. Dec. Ctrl.}, 2016, pp. 7195--7201.

\bibitem{auer2002using}
P.~Auer, ``Using confidence bounds for exploitation-exploration trade-offs,''
  \emph{Journal of Machine Learning Research}, vol.~3, no. Nov, pp. 397--422,
  2002.

\bibitem{BoydConvex2004}
S.~Boyd and L.~Vandenberghe, \emph{Convex optimization}.\hskip 1em plus 0.5em
  minus 0.4em\relax Cambridge Univ. Press, 2004.

\bibitem{gurobi}
\BIBentryALTinterwordspacing
{Gurobi Optimization LLC}, ``Gurobi optimizer reference manual,'' 2018.
  [Online]. Available: \url{http://www.gurobi.com}
\BIBentrySTDinterwordspacing

\bibitem{cvxpy}
S.~Diamond and S.~Boyd, ``{CVXPY}: A {P}ython-embedded modeling language for
  convex optimization,'' \emph{Journal of Machine Learning Research}, vol.~17,
  no.~83, pp. 1--5, 2016.

\bibitem{domahidi2013ecos}
A.~Domahidi, E.~Chu, and S.~Boyd, ``Ecos: An {SOCP} solver for embedded
  systems,'' in \emph{Proc. European Ctrl. Conf.}\hskip 1em plus 0.5em minus
  0.4em\relax IEEE, 2013, pp. 3071--3076.

\bibitem{gpyopt2016}
{The GPyOpt authors}, ``{GPyOpt}: A {B}ayesian optimization framework in
  {P}ython,'' \url{http://github.com/SheffieldML/GPyOpt}, 2016.

\bibitem{bayen2007aircraft}
A.~Bayen, I.~Mitchell, M.~Oishi, and C.~Tomlin, ``Aircraft autolander safety
  analysis through optimal control-based reach set computation,'' \emph{J.
  Guid., Ctrl. \& Dyn.}, vol.~30, no.~1, pp. 68--77, 2007.

\bibitem{stein1970singular}
E.~M. Stein, \emph{Singular integrals and differentiability properties of
  functions}.\hskip 1em plus 0.5em minus 0.4em\relax Princeton Univ. Press,
  1970, vol.~30.

\end{thebibliography}
\begin{IEEEbiography}[{\includegraphics[width=1in,height=1.25in,clip,keepaspectratio]{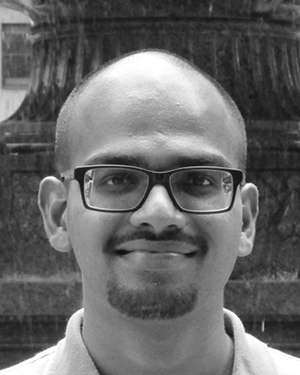}}]{Abraham Vinod}
    Abraham Vinod joined Mitsubishi Electric Research
    Laboratories as a Research Scientist in Fall 2019. He 
    received his Ph.D. degree in Electrical
    Engineering from the University of New Mexico in 2018,
    and held a postdoctoral position at the University
    of Texas at Austin. His research interests are
    in the areas of optimization, stochastic control, and
    learning. 
\end{IEEEbiography}
\begin{IEEEbiography}[{\includegraphics[width=1in,height=1.25in,clip,keepaspectratio]{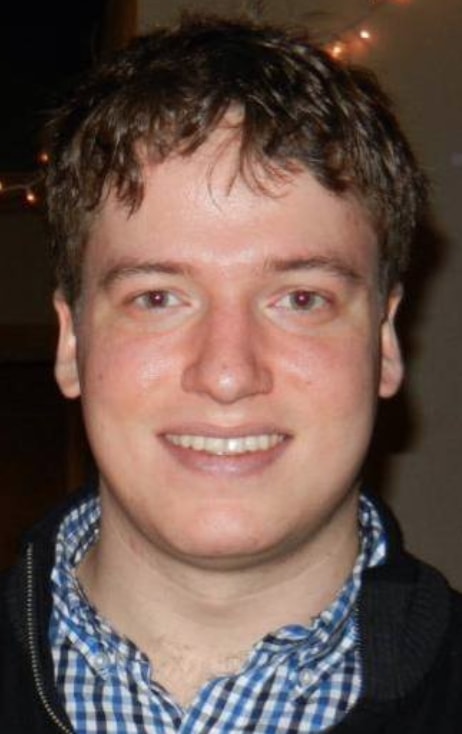}}]{Arie Israel}
Arie Israel joined the Department of Mathematics at the University of Texas at Austin as an assistant professor in Fall 2014. He received his Ph.D. degree from Princeton University in 2011, and held a postdoctoral position at New York University. His primary research focus has been on the theoretical and algorithmic foundations of extension and interpolation problems in smooth function spaces. His work draws on tools from Harmonic Analysis and PDE. Recently he has been exploring applications of his work to machine learning and control theory. 
\end{IEEEbiography}
\begin{IEEEbiography}[{\includegraphics[width=1in,height=1.25in,clip,keepaspectratio]{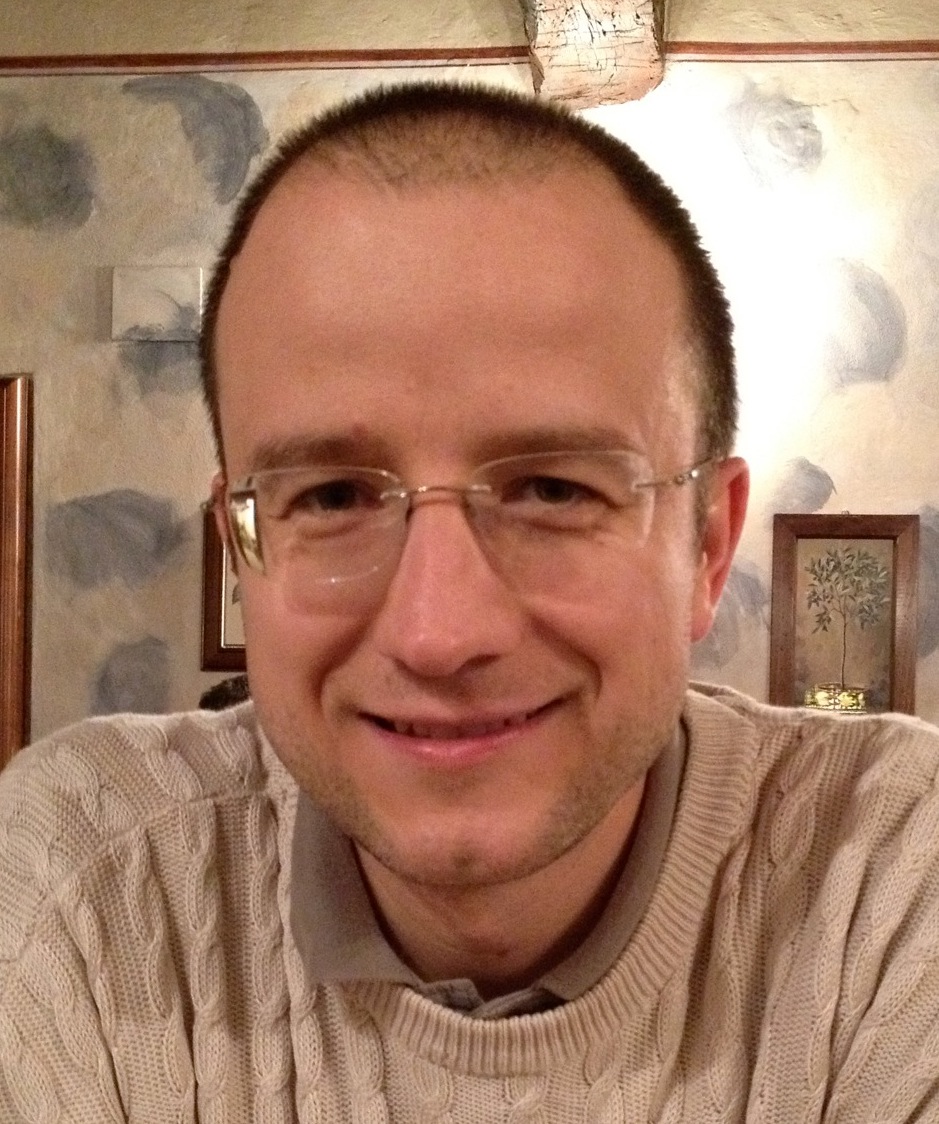}}]{Ufuk Topcu}
Ufuk Topcu is an associate professor at The University of Texas at Austin. He holds the W.A. ``Tex'' Moncrief, Jr. Professorship in Computational Engineering and Sciences. He received his Ph.D. from the University of California, Berkeley in 2008. Ufuk held a postdoctoral research position at California Institute of Technology until 2012 and was a research assistant professor at the University of Pennsylvania until 2015. His research focuses on the theoretical, algorithmic and computational aspects of design and verification of autonomous systems through novel connections between formal methods, learning theory and controls. 
\end{IEEEbiography}
\end{document}